\documentclass[11pt]{amsart}
\usepackage[paper=a4paper,                     
            left=27mm,                         
            right=27mm,                        
            top=30mm,                          
            bottom=30mm,                       
            bindingoffset=0mm                  
           ]{geometry}
\usepackage{amsmath,amssymb,amsthm}
\usepackage{lmodern}
\usepackage{mathrsfs} 
\usepackage{extarrows} 
\usepackage{hyperref} 
\usepackage{tikz}
\usepackage[all]{xy}
\usepackage{tikz-cd}
\usetikzlibrary{decorations.pathmorphing}
\usepackage{graphicx}
\usepackage[utopia]{mathdesign}
\usepackage{amsfonts,mathtools}
\usepackage{thmtools}
\usepackage{geometry,graphicx}
\usepackage{enumerate}
\usepackage{multicol}  
\usepackage{tikz}
\usepackage{tikz}
\usepackage{tikz-cd}

\newtheorem{theorem}{Theorem}[section]

\newtheorem{proposition}[theorem]{Proposition}
\newtheorem{corollary}[theorem]{Corollary}
\newtheorem{lemma}[theorem]{Lemma}
\theoremstyle{definition}
\newtheorem{definition}[theorem]{Definition}
\newtheorem{example}[theorem]{Example}

\newtheorem{remark}[theorem]{Remark}
\newtheorem{conjecture}[theorem]{Conjecture}

\newtheorem{claim}[theorem]{Claim}

\theoremstyle{remark}
\newtheorem{examples}[theorem]{Examples}
\newtheorem{question}[theorem]{Question}

\newcommand{\bL}{\mathbf{L}}
\newcommand{\bR}{\mathbf{R}}
\newcommand{\GL}{\widetilde{\textnormal{GL}}^{+}(2,\mathbb{R})}

\newcommand{\CC}{\mathbb{C}}

\newcommand{\PP}{\mathbb{P}}
\newcommand{\RR}{\mathbb{R}}

\newcommand{\ZZ}{\mathbb{Z}}
\newcommand{\NN}{\mathbb{N}}
\newcommand{\QQ}{\mathbb{Q}}

\newcommand{\cA}{\mathcal{A}}
\newcommand{\cB}{\mathcal{B}}

\newcommand{\cD}{\mathcal{D}}
\newcommand{\cE}{\mathcal{E}}
\newcommand{\cT}{\mathcal{T}}
\newcommand{\cF}{\mathcal{F}}
\newcommand{\cM}{\mathcal{M}}

\newcommand{\cO}{\mathcal{O}}

\newcommand{\cP}{\mathcal{P}}
\newcommand{\cQ}{\mathcal{Q}}
\newcommand{\cS}{\mathcal{S}}

\newcommand{\bcla}{\begin{claim}}
\newcommand{\ecla}{\end{claim}}
\newcommand{\bp}{\begin{proposition}}
\newcommand{\ep}{\end{proposition}}
\newcommand{\brem}{\begin{remark}}
\newcommand{\erem}{\end{remark}}
\newcommand{\bd}{\begin{definition}}
\newcommand{\ed}{\end{definition}}
\newcommand{\bl}{\begin{lemma}}
\newcommand{\el}{\end{lemma}}
\newcommand{\bh}{\begin{hecho}}
\newcommand{\eh}{\end{hecho}}
\newcommand{\bq}{\begin{question}}
\newcommand{\eq}{\end{question}}
\newcommand{\bo}{\begin{obs}}
\newcommand{\eo}{\end{obs}}
\newcommand{\bc}{\begin{corollary}}
\newcommand{\ec}{\end{corollary}}
\newcommand{\bcon}{\begin{conjecture}}
\newcommand{\econ}{\end{conjecture}}
\newcommand{\bnot}{\begin{notation}}
\newcommand{\enot}{\end{notation}}
\newcommand{\bdem}{\begin{proof}}
\newcommand{\edem}{\end{proof}}
\newcommand{\benum}{\begin{enumerate}}
\newcommand{\eenum}{\end{enumerate}}
\newcommand{\bitem}{\begin{itemize}}
\newcommand{\eitem}{\end{itemize}}
\newcommand{\bes}{\begin{examples}}
\newcommand{\ees}{\begin{examples}}
\newcommand{\be}{\begin{example}}
\newcommand{\ee}{\end{example}}
\newcommand{\bt}{\begin{theorem}}
\newcommand{\et}{\end{theorem}}
\newcommand{\lin}{\langle}
\newcommand{\rin}{\rangle}

\DeclareMathOperator{\rec}{rec}

\DeclareMathOperator{\Aut}{Aut}
\DeclareMathOperator{\perf}{perf}

\DeclareMathOperator{\Sets}{Sets}
\DeclareMathOperator{\Quot}{Quot}
\DeclareMathOperator{\Sub}{Sub}
\DeclareMathOperator{\Sch}{Sch}
\DeclareMathOperator{\pug}{pug}
\DeclareMathOperator{\Sym}{Sym}
\DeclareMathOperator{\Gpds}{Gpds}
\DeclareMathOperator{\rk}{rk}
\DeclareMathOperator{\QCoh}{QCoh}
\DeclareMathOperator{\Coh}{Coh}

\DeclareMathOperator{\Hom}{Hom}
\DeclareMathOperator{\Ext}{Ext}
\DeclareMathOperator{\Stab}{Stab}
\DeclareMathOperator{\Rep}{Rep}

\DeclareMathOperator{\Mor}{Mor}

\DeclareMathOperator{\op}{op}
\DeclareMathOperator{\gl}{gl}
\DeclareMathOperator{\Pic}{Pic}

\DeclareMathOperator{\rank}{rank}

\title{Moduli of Bridgeland semistable holomorphic triples}
\author{DOMINIC BUNNETT}
\address{Technische Universität Berlin, Straße des 17. Juni 135, \newline \indent Raum 613,
Berlin 10623, Germany}
\email{bunnett@math.tu-berlin.de}
\urladdr{http://page.math.tu-berlin.de/~bunnett/}
\author{ALEJANDRA RINC\'{O}N-HIDALGO}
\address{ICTP, Strada Costiera 11, room 126, 34151 Trieste, Italy.}
\email{arincon@ictp.it}
\urladdr{http://users.ictp.it/~arincon/}

\begin{document}
\maketitle

\begin{abstract}
We prove that the moduli stack of Bridgeland semistable holomorphic triples over a curve of $g(C)\geq 1$ with a fixed numerical class and phase is an algebraic stack of finite type over $\CC$ and admits a proper good moduli space. We prove that this also holds for a class of Bridgeland stability conditions on the category of holomorphic chains $\cT_{C,n}$.

In the process, we construct an explicit geometric realisation of $\cT_{C,n}$ and prove the open heart property for noetherian hearts in admissible categories of $D^b(X)$, where $X$ is a smooth projective variety over $\CC$, whose orthogonal complements are geometric triangulated categories.
\end{abstract}

\section{Introduction}
The purpose of this paper is to study the moduli of Bridgeland semistable holomorphic chains.

Holomorphic chains were first introduced by \'Alvarez-Cons\'ul and Garc\'ia-Prada in \cite{ACGP01}.
An $n$-holomorphic chain is a chain of morphisms
\[E_1 \xrightarrow{\,\,\varphi_1 \,\,} E_2 \xrightarrow{\,\,\varphi_2 \,\,} \cdots \xrightarrow{\varphi_{n-1}} E_{n}\enspace ,\]
where $E_i \in \Coh(X)$ with $X$ a smooth projective variety over $\CC$.
Moduli spaces of holomorphic chains of vector bundles were constructed by Schmitt \cite{S03} using geometric invariant theory (GIT).
These moduli spaces have played an important role in the study of Higgs bundles \cite{BGPG03,GPH13}.

We denote the abelian category of holomorphic chains by $Q_{X,n}$ and its derived category $\cT_{X,n}$.
We also refer to objects of $\cT_{X,n}$ as holomorphic chains.
When $n=2$ we write $\cT_{X} \coloneqq \cT_{X,2}$ and refer to it as the category of holomorphic triples. 

Stability conditions on triangulated categories were introduced by Bridgeland \cite{B07} and play a very important role in algebraic geometry via the study of moduli spaces and wall-crossings.
In \cite{MRRHR20} the second author, Mart\'inez-Romero and R\"uffer completely described the stability manifold of $\cT_C$ for a curve $C$.

The construction of moduli spaces of Bridgeland semistable objects is highly non-trivial.
Building on the work of Lieblich \cite{L05}, Toda carried out this construction for the derived category of a K3 surface $X$ \cite{T08}.
In particular, Toda proved that the moduli space of $\sigma$-semistable objects in $D^b(X)$ is an algebraic stack of finite type over $\CC$ for some $\sigma \in \Stab(X)$.

Toda and Piyaratne conjectured in \cite[Conjecture 1.1]{PT19} that the same holds for the derived category of any smooth projective variety.
This conjecture has been confirmed for K3 surfaces, threefolds satisfying the Bogomolov-Gieseker inequality, and where replaces the derived category with the Kuznetsov component of a cubic fourfold \cite{T08,PT19,BLMNPS20}.
The main result of this paper confirms this conjecture when $D^b(X)$ is replaced by $\cT_C$.

Moduli spaces of objects in $D^b(X)$ were studied by Lieblich in \cite{L05} and by Abramovich and Polishchuk in \cite{AP06}. 
In \cite{BLMNPS20}, motivated by the study of the Kuznetsov component, Bayer, Lahoz, Macr\`i, Nuer, Perry and Stellari studied moduli problems associated to full admissible subcategories of $D^b(X)$ and of relative moduli spaces building on work of Lieblich, Piyaratne, and Toda \cite{L05,T08,PT19}.

In order to construct the moduli space of Bridgeland semistable holomorphic chains, we first need to embed $\cT_{X,n}$ into $D^b(Y_{X,n})$ for a smooth projective variety $Y_{X,n}$.
This is precisely the contents of Theorem \ref{theorem_gr_quivers} which constructs an explicit embedding.
We refer to this embedding by $\cT_{X,n} \hookrightarrow D^b(Y_{X,n})$ as a geometric realisation of $\cT_{X,n}$.

The construction of $Y_{X,n}$ is a generalisation of the work of Orlov \cite{O15}, the key difference being that semiorthogonal decompositions are needed in place of strong exceptional collections. 
In constructing the embedding we characterise $\cT_{X,n}$ via its semiorthogonal components and the associated gluing functor.
Gluing semiorthogonal components (as in \cite{KL15}) requires that one works on the level of dg-categories.

To understand families of objects in $\cT_{X,n}$ we look to the base change
\[\left(\cT_{X,n}\right)_S \subset D^b(Y_{X,n} \times S)\enspace ,\]
where $S$ is a base scheme, defined by Kuznetsov \cite{K11}.
We then consider local t-structures \cite{BLMNPS20} on admissible subcategories in $D^b(X)$ (so-called sheaves of t-structures by Abramovich and Polishchuk \cite{AP06}).

As laid out in \cite[Section 6.1]{AP06}, to construct moduli spaces of Bridgeland semistable objects with respect to an algebraic stability conditions three problems remain:
\begin{enumerate}
    \item the \emph{generic flatness property};
    \item the \emph{open heart property}, and;
    \item the \emph{boundedness} of semistable objects of fixed type $\beta$ and phase $\phi.$ 
\end{enumerate}

Let us first address boundedness.
If boundedness of a moduli space associated to a stability condition can be shown to hold, then the same is true for any stability condition in its connected component.
Thus, via the well-studied GIT moduli space of holomorphic chains of vector bundles \cite{S03}, we can conclude boundedness for the connected component containing GIT-stability conditions.
Moreover, by \cite[Theorem 1.1]{MRRHR20}, boundedness follows for the entire stability manifold when $n=2$.

For a noetherian heart in $D^b(X)$ the `open heart property' (Definition \ref{def:OHC}) is already known \cite[Proposition 3.3.2]{AP06}.
In the case of $\cT_{X,n}$, we appeal to the structure inherited from the geometric realisation.
For a noetherian heart $\cA\subseteq \cT_{X,n},$ our approach is to pass down properties already proved for hearts in $D^b(X).$
To this end, we construct a bigger noetherian heart via recollement (as in \cite{BBD}) $\widetilde{\cA}\subseteq D^b(Y_{X,n})$ containing $\cA.$ 
In Proposition \ref{prop:rec_base_change}, we prove in a general setting, that being a recollement heart is stable under base change.

More generally, we use the same strategy to prove the open heart property for a heart $\cA$ in an admissible subcategory $\cT \subseteq D^b(X)$ under the condition that $\cT^{\perp}$ is geometric (see Remark \ref{rmk:OHC-general}).

We first prove the generic flatness property (Definition \ref{def:GF}) for algebraic stability conditions constructed via gluing (Remark \ref{Rem_filtrationCP}).
By adapting the results of \cite{T08}, we can then extend this to any algebraic stability condition in the $\GL$-orbit of a gluing stability condition.
Furthermore, given that the support property is satisfied for triples, we obtain the following theorem.

\begin{theorem}[Theorem \ref{thm_main1} and \ref{thm:good-moduli}]
Let $\sigma\in \Stab(\cT_{C})$ be a stability condition.
The stack $\cM^{\beta, \phi}(\sigma)$ is an algebraic stack of finite type over $\CC$ admitting a proper good moduli space.
\end{theorem}
Further, we obtain a partial result in the case of holomorphic chains in Proposition \ref{prop_algstack_chains} and Theorem \ref{thm:good-moduli}.

For holomorphic triples, one has a complete picture of the wall and chamber structure in the classical situation \cite{S18}.
By studying the wall and chamber decomposition of the stability manifold, we get a full picture of the moduli of holomorphic triples.
We expect that the moduli spaces are projective and their that their birational geometry is dictated by the wall and chamber structure - this will be pursued elsewhere.

\subsection*{Layout}
The layout of the paper is as follows.
In Section 2, we provide preliminaries, fixing notation and providing results central to this paper.
In Section 3, we study geometric realisations of $\cT_{X,n}$. This involves first studying the structure of $\cT_{X,n}$ in the setting of dg-categories and subsequently explicitly constructing a geometric realisation as a tower of projective bundles with foundation $X$.
Finally, in Section 4 we study the moduli stacks themselves.

\subsection*{Notation and conventions}
We always work over $\CC$. 
A curve is a smooth irreducible projective variety of dimension 1.
We denote by $\CC$-dgm the dg-category of complexes of $\CC$-vector spaces.


\section{Preliminaries}
\subsection{Bridgeland stability conditions}

\bd 
A \emph{t-structure} on a triangulated category $\cT$ consists of a pair of full additive subcategories  $(\cT^{\leq 0},\cT^{\geq 0})$ satisfying the following properties.
We write $\cT^{\leq i} \coloneqq \cT^{\leq 0}[-i]$ and $\cT^{\geq i} \coloneqq \cT^{\geq 0}[-i]$ for $i \in \ZZ$.
\begin{enumerate}
\item $\Hom_\cT(\cT^{\leq 0}, \cT^{\geq 1})=0$.
\item For all $E \in \cT$, there is a distinguished triangle $G \rightarrow E \rightarrow F \rightarrow G[1]$
 with $G \in \cT^{\leq 0}$ and $F \in \cT^{\geq 1}$.
\item $\cT^{\leq 0} \subset \cT^{\leq 1}$ and $\cT^{\geq 0} \supset \cT^{\geq 1}$.
\end{enumerate}
A t-structure is \emph{bounded} if every $E\in \cT$ is contained in $\cT^{\leq n} \cap \cT^{\geq -n}$ for some $n>0.$
The \emph{heart} of a bounded t-structure $(\cT^{\leq 0},\cT^{\geq 0})$ is defined as $\cA\coloneqq \cT^{\leq 0}\cap \cT^{\geq 0}.$
\ed

\brem The heart $\cA$ of a bounded t-structure on $\cT$ is an abelian category and $K(\cA)= K(\cT)$.
\erem

\bd \label{Slicing} A \emph{slicing} $\mathcal{P}$ on $\cT$ is a collection of full subcategories $\mathcal{P}(\phi)$ for all $\phi \in \RR$ satisfying:\begin{enumerate}\item $\mathcal{P}(\phi)[1]=\mathcal{P}(\phi + 1)$, for all $\phi \in \RR$.
\item If $\phi_1 > \phi_2$ and $E_i \in \mathcal{P}(\phi_i)$, $i=1,2$, then $\Hom_{\cT} (E_1,E_2)=0$.
\item For every nonzero object $E \in \cT$ there exists a finite sequence of maps
$$0=E_0\xrightarrow{f_{0}}E_1\xrightarrow{f_1}\dots \rightarrow E_{m-1}\xrightarrow{f_{m-1}}E_m=E$$
and of real numbers $\phi_0 > \cdots > \phi_{m-1}$ such that the cone of $f_j$ is in $\mathcal{P}(\phi_j)$ for $j=0,\cdots,{m-1}$.
\end{enumerate}
\ed

For every interval $I\subseteq \RR$ we define $\mathcal{P}(I)$ to be the extension-closed subcategory generated by the subcategories $\mathcal{P}(\phi)$ with $\phi\in \RR.$

\bd
Let $\cA$ be a heart.
We say that a group homomorphism $Z\colon K (\cA)\rightarrow \CC$ is a \emph{stability function} on $\cA$ if the image of $Z$ is contained in the semi-closed upper half plane
$\overline{\mathbb{H}}=\{\alpha\in \CC\mid \Im(\alpha)\geq 0 \textnormal{ and if } \Im(\alpha)=0\textnormal{, then } \Re(\alpha)<0 \}$. 
\ed
We now fix a finite rank $\mathbb{Z}$-lattice $\Lambda$ and a surjective homomorphism $v\colon K(\cT)\twoheadrightarrow \Lambda$.
When $\cT$ is numerically finite, we have that the numerical Grothendieck group $N(\cT)$ is a finite rank $\ZZ$-lattice.
We often choose $\Lambda=N(\cT)$ and $v$ as the natural projection.

We consider a group homomorphism $Z\colon \Lambda \rightarrow \CC$, such that $Z \circ v \colon K(\cA)\rightarrow \CC$ is a stability function on $\cA$. We define the slope by
\begin{equation} \nonumber
\mu_{\sigma}(E)=
\begin{cases} \nonumber
-\frac{\Re(Z(E))}{\Im(Z(E))} & \textnormal{ if } \Im(Z(E))\neq 0 \\
+\infty & \textnormal{otherwise} \enspace ,
\end{cases}
\end{equation}
where $Z(E)\coloneqq Z(v([E]))$.
We say that a non-zero $E\in \cA$ is \emph{$\sigma$-semistable (stable)} if for all proper subobjects  $F\subseteq E$, we have that $\mu_{\sigma}(F)\leq\mu_{\sigma}(E) (<)$.
We also define the phase of $E$ as  $\phi(E)=\arg(Z(E))\frac{1}{\pi}\in (0,1]$.  

\bd \label{DefSC}
A \emph{pre-stability condition} on $\cT$ is a pair $\sigma=(Z,\cA)$, where $\cA\subseteq \cT$ is the heart of a bounded t-structure and $Z\colon \Lambda \rightarrow \CC$ is a group homomorphism such that $ Z\circ v\colon K(\cA)\rightarrow \CC$ is a stability function on $\cA$ and every $E\in \cA$ has a Harder-Narasimhan (HN) filtration with $\sigma$-semistable factors.
If additionally $\sigma$ satisfies the support property i.e.\ there is a symmetric bilinear form $Q$ on $\Lambda_{\RR} \coloneqq \Lambda\otimes \RR$ such that $Q(v(E),v(E))\geq 0$ for all $\sigma$-semistable objects $E\in \cA$ and it is negative definite on the kernel of $Z,$ then $\sigma$ is called a \emph{Bridgeland stability condition} with respect to $\Lambda$.
\ed

\brem[{\cite[Proposition 5.3]{B07}}]
To give a pre-stability condition $\sigma$ on $\cT$ is equivalent to giving a slicing $\mathcal{P}$ and a group homomorphism $Z\colon \Lambda\rightarrow \CC$ such that for every non-zero $E\in \mathcal{P}(\phi),$ we have that $Z(E) \in \RR_{>0} \cdot e^{i \pi \phi}.$
The objects of $\mathcal{P}(\phi)$ are precisely the $\sigma$-semistable objects of phase $\phi.$
\erem

The set of Bridgeland stability conditions with respect to $(\Lambda,v)$ is denoted by $\Stab_{\Lambda}(\cT)$ and moreover, $\Stab_{\Lambda}(\cT)$ admits the structure of a complex manifold \cite{B07} and is referred to as the stability manifold.
If $\Lambda=N(\cT)$ and $v$ the natural projection, then the set of stability conditions is denoted by $\Stab(\cT)$.
If $\cT = D^b(X)$, we write by $\Stab(X)$.

\bd
We call a stability condition $\sigma=(Z,\cA) \in \Stab(\cT)$ \emph{algebraic} if the image of $Z \colon N (\cT) \rightarrow \CC$ is contained in $\QQ \oplus \QQ i$.
\ed

As in \cite[Lemma 8.2]{B07}, we consider the right action of $\GL$ on the stability manifold.
If $\sigma=(Z,\cA)$ is a stability condition and $g=(T,f)\in \GL$, then we define $\sigma \cdot g=(Z',\mathcal{P}')$ to be $Z=T^{-1} \circ Z$ and $\mathcal{P}'(\phi)=\mathcal{P}(f(\phi))$, where $\mathcal{P}$ and $\mathcal{P}'$ are the slicings of $Z$ and $Z'$ respectively.
Note that the $\GL$-action preserves the semistable objects, but relabels their phases.

Note that by \cite[Theorem 2.7]{M07}, if $g(C)\geq 1$ then $\Stab(C)\cong \GL.$ 
Let us consider the group $\Aut_{\Lambda}(\cT)$ of autoequivalences $\Phi$ on $\cT$ whose induced automorphism $\phi_*$ of $K(\cT)$ is compatible with the map $v\colon K(\cT)\rightarrow \Lambda.$
We define a left action of the group $\Aut_{\Lambda}(\cT)$ on the set of stability conditions.
For $\Phi\in\Aut_{\Lambda}(\cT)$ of $\cT.$ 
We define $\Phi(\sigma)=(Z',\mathcal{P}')$  as  $Z'=Z\circ \phi^{-1}_*$ and $\mathcal{P}'(\phi)=\Phi(\mathcal{P}(\phi)).$
Note that if $E$ is a $\sigma$-semistable object, then $\Phi(E)$ is $\Phi(\sigma)$-semistable.

\subsection{Bridgeland stability conditions on the category of holomorphic chains}

Let $X$ be a smooth projective variety over $\CC.$
An $n$-holomorphic chain is a chain of morphisms
\[E_1 \xrightarrow{\,\,\varphi_1 \,\,} E_2 \xrightarrow{\,\,\varphi_2 \,\,} \cdots \xrightarrow{\varphi_{n-1}} E_{n}\enspace .\]
where $E_i \in \Coh(X).$  
We denote the abelian category of such chains by $Q_{X,n}$ and its derived category by
\[\cT_{X,n} := D^b(Q_{X,n})\enspace .\]
The special case $Q_{X,2}$ is the abelian category of \emph{holomorphic triples over} $X$ as in \cite{MRRHR20}.
In this case, we write $\cT_{X}:=\cT_{X,2}$.

Recall the description of $\cT_X$ given in \cite[Section 3.1]{MRRHR20}:
There is a semiorthogonal decomposition $\cT_{X}=\lin D_1,D_2\rin$ where $D_j\cong \cT_{X}$ is the image of the fully faithful embeddings
\begin{multicols}{2}\noindent
\begin{eqnarray}\nonumber
i_1\colon D^b(X) &\hookrightarrow &\cT_{X}\\
 E&\mapsto & (E\rightarrow 0) \enspace ,\nonumber
\end{eqnarray}
\begin{eqnarray} \nonumber
i_{2}\colon  D^b(X) &\hookrightarrow &\cT_{X}\\
 E&\mapsto & (0\rightarrow E)\nonumber
\end{eqnarray}
\end{multicols}
respectively.

Note that we  have a semiorthogonal decomposition of the form $\cT_{X,n}=\lin D_1,\cT_{X,n-1}\rin$ by seeing $\cT_{X,n-1}$ as the subcategory of objects $E\in \cT_{X,n}$ with $E_1=0$.
Inductively, we get the \emph{standard semiorthogonal decomposition of} $\cT_{X,n}$.

\bd\label{def:chain-semiorth-decomp}
We denote the semiorthogonal decomposition of $\cT_{X,n}$ by
\[\cT_{X,n} = \lin D_{1},\dots,D_{n}\rin \enspace ,\]
where $D_{j}\cong D^b(X)$ and is given by the image of the functor $i_{j}\colon D^b(X) \hookrightarrow \cT_{X,n}$ defined by sending $E$ to the chain satisfying $E_j=E$ and $E_l=0$ for $l\neq k.$
\ed

\brem 
From these semiorthogonal decompositions, we conclude that $K(\cT_{X,n})=\bigoplus^n_{i=1} K(D^b(X)).$
Moreover, if $C$ is a  curve, we have that $N(\cT_{C,n})=\ZZ^{2n}$ with the isomorphism given by sending $[E]$ to $(d_1,r_1,\dots, d_n,r_n)$ with $d_i\coloneqq \deg(E_i)$ and $r_i=\rank(E_i).$
\erem

\subsubsection{CP-gluing and recollement} \label{sec:CP_rec}
 Let $\mathcal{T}$ be a triangulated category equipped with a semiorthogonal decomposition $\mathcal{T}=\langle \mathcal{D}_1, \mathcal{D}_2\rangle$ and let $i_j\colon D_{j}\rightarrow \cT$ be full embeddings for $j=1,2.$  Throughout the whole paper, we assume that semiorthogonal components are \emph{admissible} i.e.\ the functors $i_j$ have a left adjoint $i^*_j\colon \cT \rightarrow D_j$ and a right adjoint  $i^!_j\colon \cT \rightarrow D_j.$ 
 
\begin{proposition}[{\cite[Lemma 2.1]{CP10}}]\label{CP}
With the above notations, assume that we have t-structures $(\mathcal{D}_{i}^{\leq 0},\mathcal{D}_{i}^{\geq 0})$ with hearts $\mathcal{A}_i$ in $\mathcal{D}_i$, for $i=1,2$, such that
\begin{equation}\label{eq:glHom0}
\Hom_{\mathcal{D}}^{\leq 0}(i_1 \mathcal{A}_1,i_2 \mathcal{A}_2)=0 \enspace .
\end{equation}
Then there is a t-structure on $\mathcal{T}$ with the heart
\begin{equation}\label{eq:gluedA}
\gl( \mathcal{A}_1, \mathcal{A}_2)=\{E \in \mathcal{D} \mid i^!_2 E \in \mathcal{A}_2, i^*_1 E \in \mathcal{A}_1 \} \enspace .
\end{equation}
Moreover, $i_k\mathcal{A}_k \subset \mathcal{A} \coloneqq \gl(\mathcal{A}_1, \mathcal{A}_2)$ for $k=1,2$.
\end{proposition}

Let us consider $\cT=\lin D_1,\dots, D_n \rin$ and the full embeddings $i_j\colon D_j\rightarrow \cT$.
Let $\cA_i\subseteq D_{i}$ be the hearts of a bounded t-structures for $i=1,\dots ,n$.
Assume the hearts satisfy the \emph{gluing condition}
\[\Hom_{\cT}^{\leq 0}(i_{l}\cA_{l},i_{j}\cA_{j})=0\]
for all $0<l<j\leq n$.
We define $\cB_1\coloneqq \gl(\cA_{n-1},\cA_{n})$ and $\cB_{j}\coloneqq \gl(\cA_{n-{j}},\cB_{j-1})$.
By Proposition \ref{CP}, we have that $\cB_{j}$ is a heart of a bounded t-structure on the triangulated subcategory $\lin D_{n-j},\dots ,D_{n} \rin \subseteq \cT$ for $j=1,\dots, n-1$.
 
\bd \label{definition_CPheart}
We define $\gl(\cA_{1},\dots, \cA_{n})\coloneqq \cB_{n-1}\subseteq \cT.$
We refer to $\gl(\cA_{1},\dots, \cA_{n})$ as a \emph{gluing heart} of $\cT$ with respect to the semiorthogonal decomposition $\cT=\lin D_1,\dots, D_n \rin$.
\ed
 
\brem \label{Rem_filtrationCP}
 \begin{enumerate}
    \item Since the gluing of two noetherian hearts is again a noetherian heart, it follows that a gluing heart in the sense of Definition \ref{definition_CPheart} is noetherian.
     \item Note that if $E\in \cA=\gl(\cA_1,\dots,\cA_n)$ if and only if there is a sequence of triangles
 \[\begin{tikzpicture}[description/.style={fill=white,inner sep=1.4pt}]
    \matrix (m) [matrix of math nodes, row sep=1.5em,
    column sep=0.5em, text height=0.8ex, text depth=0.19ex]
    {0 & & E_1  & & E_2 \cdots & & E_{n-1} & & E \\
& A_n & & A_{n-1} & & \cdots  & & A_1  \\ };
       \path[->]
       
       (m-1-7) edge node[auto] {} (m-1-9)
       (m-1-5) edge node[auto] {} (m-1-7)
           (m-1-3) edge node[auto] {} (m-1-5)
		   (m-1-1) edge node[auto] {} (m-1-3)
         (m-1-5)   edge node[auto] {} (m-2-4)
                  (m-1-3)   edge node[auto] {} (m-2-2)
    (m-2-2) edge[dashed] node[auto] {} (m-1-1)
(m-2-4) edge[dashed] node[auto] {} (m-1-3)
(m-1-9) edge node[auto] {} (m-2-8)
(m-2-8) edge[dashed] node[auto] {} (m-1-7);
\end{tikzpicture}\enspace .\]
with $A_j\in i_{j}\cA_j,$ for $j=1,\dots, n.$
\item $\gl(\cA_1,\cA_2,\cA_3)=\gl(\gl(\cA_1,\cA_2),\cA_3)$ in $\cT=\lin D_1,D_2,D_3 \rin.$ 
\end{enumerate}
 \erem 

Let $\sigma_{i}=(Z_{i},\cA_{i})\in \Stab(X).$
If the hearts $\cA_{i}$ satisfy the gluing condition, then by recursively applying Proposition \ref{CP}, there is a stability function on $\cA=\gl(\cA_{1},\dots, \cA_{n})$ given by $Z(E)=\sum^n_{i=1}Z_{i}(A_{i}),$ with $A_{i}$ as in Remark \ref{Rem_filtrationCP}.
Moreover, if $\sigma_{i}$ are algebraic stability conditions, we obtain that $\gl(\sigma_{1},\dots, \sigma_{n})\coloneqq(Z,\cA)$ is a locally finite pre-stability condition, see \cite[Lemma 4.4]{B08}.
If a stability condition $\sigma\in \Stab(\cT_{X,n})$ is equal to $\gl(\sigma_1,\dots ,\sigma_{n})$, we refer to $\sigma$ as a \emph{gluing stability condition}.

The next example defines $\alpha$-stability as studied in \cite{ACGP01}.

\be\label{ex_alpha}($\alpha$-stability)
Let $C$ be a curve and $(\alpha_{j})_{j=1,\dots,n}\in\QQ^n.$
Consider $\sigma_{\alpha_j}=(\Coh(C),Z_{\alpha_j})\in \Stab(C)$ where $Z_{\alpha_j}(d,r)=-d-\alpha_jr+ir$.
We define $\sigma_{\alpha}=\gl(\sigma_{\alpha_{1}},\dots ,\sigma_{\alpha_{n}})$ is a locally finite pre-stability condition.
\ee

\begin{theorem}[{\cite[Theorem 1.4.10]{BBD}}]\label{thm:BBDheart}
Let $(D^{\leq 0}_1,D^{\geq 0}_1)$ and 
$(D^{\leq 0}_2,D^{\geq 0}_2)$ be the bounded t-structures with hearts $\cA_{1}$ and $\cA_{2}$ respectively and $\cT=\lin D_1, D_2\rin.$

Then there is a t-structure $(\cT^{\leq 0}, \cT^{\geq 0})$ in $\mathcal{T}$ defined by:
\[\begin{array}{c}
\mathcal{T}^{\leq 0}\coloneqq \{T \in \mathcal{T} \mid i_1^*T \in D_1^{\leq 0},{i_2}^* T \in D_2^{\leq 0} \} \\
\mathcal{T}^{\geq 0}\coloneqq \{T \in \mathcal{T} \mid i_1^*T \in D_1^{\geq 0}, {i_2}^! T \in D_2^{\geq 0} \} \enspace .
\end{array}\]
We denote by $\rec(\cA_1,\cA_2) \coloneqq \mathcal{T}^{\leq 0} \cap \mathcal{T}^{\geq 0}$. We say that $\rec(\cA_1,\cA_2)$ is a recollement heart with respect to the semiorthogonal decomposition $\cT=\lin D_1, D_2\rin$.
\end{theorem}

\brem \label{rem_glu_rec}
By \cite[Proposition 2.8.12]{MR18}, if $\cA_j\subseteq D_{j}$ satisfy the gluing condition, then $\rec(\cA_1,\cA_2)=\gl(\cA_1,\cA_2)$.
\erem 

Let $C$ be a curve with $g(C)\geq 1$. Under the assumption that all the pre-stability conditions constructed in \cite{MRRHR20} satisfy the support property, we have the following result.

\bt \label{them_classificationC_connected} \cite[Theorem 1.1 and Theorem 4.47]{MRRHR20}
The stability manifold $\Stab(\cT_{C})$ is a connected $4$-dimensional complex manifold. 
Moreover, if $\sigma=(Z,\cA)\in \Stab(\cT_C)$ then, up to autoequivalence, it satisfies one of the following properties:
\begin{enumerate}
    \item $\sigma$ is a gluing stability condition with respect to the semiorthogonal decomposition $\cT_{C}=\lin D_1,D_2\rin.$
    \item There is $g\in \GL$ such that, $\sigma'=\sigma g$ is a gluing stability condition in (1).
    \item There is a $g\in \GL$ and  $\sigma''=(Z'',\cA'')$ satisfying (2), such that for $\sigma g=(Z',\cA')$ we have that  $\cA''=\cA'.$
\end{enumerate}
\et

\subsection{Base change for semiorthogonal decompositions and hearts}

The base change of a semiorthogonal decomposition was introduced by Kuznetsov in \cite{K11} and generalised in \cite{BLMNPS20}.
The base change a heart of $D^b(Y)$ for a smooth projective variety was first given in \cite{AP06} and greatly generalised in the relative setting in \cite{BLMNPS20}.

\bd 
Let $\cT \subseteq  D(\QCoh(X))$ be a triangulated subcategory.
If $Y$ is a scheme and $\Phi\colon \cT \rightarrow D(\QCoh(Y))$ is a triangulated functor, we say that $\Phi$ has cohomological amplitude $[a, b]$ if
\[\Phi\left(\cT\cap D_{\rm{qc}}^{[p,q]}(X)\right) \subseteq D_{\rm{qc}}^{[p+a,q+b]}(Y)\]
for all $p,q \in \mathbb{Z}$, where $(D_{\rm{qc}}(X)^{\leq 0}, D^{\geq 0}_{\rm{qc}}(X))$ is the standard t-structure of $D(\QCoh(X))$.
We say $\Phi$ has \emph{left finite cohomological amplitude} if $a$ can be chosen finite, \emph{right finite cohomological amplitude} if $b$ can be chosen finite, and \emph{finite cohomological amplitude} if $a$ and $b$ can be chosen finite. 
We say that a semiorthogonal decomposition $\cT = \lin D_1,\dots, D_n\rin$ is of (right or left) finite cohomological amplitude if its projection functors have (right or left) finite cohomological amplitude.
\ed

\brem \label{rem_strongfca_general}
In order to apply the results of \cite{K11} and \cite{BLMNPS20}, we require a full admissible triangulated subcategory $\cT\subseteq D^b(Y)$ which is a strong semiorthogonal component of a semiorthogonal decomposition of finite cohomological amplitude.
Indeed this is the case under the condition that $Y$ is a smooth projective variety by \cite[Lemma 2.9]{K11}.
Moreover, the admissibility of $\cT$ then implies the admissibility of $\cT^{\perp}$ in this case and thus we only require that $\cT$ is admissible.
\erem

We start by giving the base change of an admissible semiorthogonal component.
Let $Y$ be a smooth projective variety and $S$ be a quasi-projective variety.
Let $p\colon X\times S \rightarrow X$ and $q\colon X\times S \rightarrow S$ be the projections. 

\bp \label{prop_base_change_sod}\cite[Corollaries 5.7 and 5.9]{K11}
Let $\cT \subseteq D^b(Y)$ be an admissible subcategory,
then the category
\[\cT_{S} = \{F \in D^b(Y\times S)  \mid  \bR p_{*}(F \otimes q^*
G) \in \widehat{\cT} \textnormal{ for all } G \in D_{\perf}(S)\} \enspace ,\] 
is an admissible subcategory in $D^b(Y\times S)$ such that the corresponding projection functor has finite cohomological amplitude.
Here $\widehat{\cT}$ is the minimal triangulated subcategory of $D(\QCoh(Y))$ containing $\cT$ closed under arbitrary direct sums.
Additionally, we have that $p^*(\cT)\subseteq \cT_{S}$ and $\bR p_*(\cT_{S})\subseteq \cT,$ if $S$ is projective. Moreover, $D^b(Y\times S)=\lin \cT^{\perp}_{S},\cT_{S} \rin.$
\ep

Let us consider the inclusion $i_{s}\colon  Y \times \{s\} \hookrightarrow Y\times S$ and we denote $E_{s}\coloneqq \bL i_{s}^*(E).$

\bl\label{lemma_monster_9.3}\cite[Lemma 9.3]{BLMNPS20}
Let $S$ be a quasi-compact $\CC$-scheme with affine diagonal and let $E \in D^b(Y \times S)$.
Then
\begin{enumerate}
    \item $E \in \cT_S$ if and only if   $E_s \in \cT$ for every $s \in S,$ where $i_{s}\colon  X \times \{s\} \hookrightarrow X\times S.$
    \item The set
  $ \left\{ s \in S \,\, | \,\, E_s \in \cT \right\}$
    is open.
\end{enumerate}
\el

\bt \label{thm_base_change_hearts} \cite[Theorem 5.7]{BLMNPS20} Let Y be a smooth projective variety and $\cT\subseteq D^b(Y)$ be a full admissible subcategory. Let $(\cT^{\leq 0}, \cT^{\geq 0})$ be a bounded t-structure on $\cT$ with a noetherian heart $\cA$. If $S$ is a smooth quasi-projective variety, then 
\begin{enumerate}
    \item (\emph{$S$-local})  For every open $U \subseteq S$, there exists a t-structure $(\cT_{U}^{\leq 0}, \cT^{\geq 0}_{U})$
on  $\cT_{U}$ with heart $\cA_{U}$ such that the restriction functor $i^*\colon \cT_{S}\rightarrow \cT_{U},$ induced by $i\colon U\hookrightarrow S,$  is t-exact.
\item $\cA_{S}$ is noetherian.
\item If $S$ is projective and  $L$ is an ample line bundle, then 
\[(\cT^{[a,b]})_{S}=\{E\in \cT_{S} \,\, | \,\, \bR p_{*}(E\otimes q^*(L)^n)\in \cT^{[a,b]} \textnormal{ for all } n\gg 0 \}\enspace .\]
\item Let $S'$ be a smooth quasiprojective variety. For a morphism $f\colon S'\rightarrow S$  and its induced morphism
$f'\colon S'\times Y \rightarrow S\times Y,$ we  have that $f'^*\colon \cT_{S}\rightarrow \cT_{S'}$ is right t-exact. Moreover, if $f$ is flat, then $f'^*$ is t-exact and if $f$ is finite, then $f'_*\colon \cT_{S'}\rightarrow \cT_{S}$ is t-exact. 
\end{enumerate}
\et


\section{Geometric realisations}

 A  \emph{geometric realisation} of a triangulated category $\cT$ is a fully faithful embedding $\cT \hookrightarrow D^b(X)$ for some $\CC$-scheme $X$.
In general, we will be interested in the case where $X$ is a smooth projective variety. In order to use \cite{BLMNPS20}, we are going to prove the existence of a geometric realisation $\cT_{X,n}\hookrightarrow D^b(Y_{X,n})$ of $\cT_{X,n}$ for a smooth projective variety $Y_{X,n}.$

\subsection{dg-enhancements and gluing}

A general reference for dg-categories is \cite{Ke06} and we refer the reader there for definitions of dg-categories and dg-bimodules.
We use the conventions and notation of \cite{KL15}.

In this section we prove that a triangulated category $\cT$ with a semiorthogonal decomposition, such that $\cT$ and the semiorthogonal components admit unique dg-enhancements, is completely determined by the semiorthogonal components and the gluing functor.
This result is well-known and is a consequence of \cite[Proposition 4.10]{KL15}.
To the best of the authors' knowledge, a proof does not occur in the literature and so one is included for the sake of completeness.

Suppose that $\cD$ is a dg-category.
We denote its homotopy category $[\cD]$.
Recall that a dg-category is \emph{pretriangulated} if $[\cD]$ is a triangulated category.

\bd 
An \emph{enhancement} of a triangulated category $\cT$ is a pretriangulated dg-category $\cD$ with an equivalence $\cT \cong [\cD]$ of triangulated categories.
\ed

A triangulated category $\cT$ has a unique dg-enhancement, if any two dg-enhancements $\cD, \cD'$ of $\cT$ are quasiequivalent.
 
\bl \label{lemma_uniquedgn}
The category $\cT_{X,n}$ has a unique dg-enhancement.
\el

We denote this enhancement $\cD_{X,n}$.
The existence follows from standard arguments, see \cite{CS17}.
If $\cD(X)$ is the dg-enhancement of $D^b(X).$

It could be also explicitly given by $\Mor(\cD(X))$, the dg-category of morphisms as defined in \cite[Example 2.9]{D04}.

\begin{proof} 
We proceed by induction on $n$.
For $n=1,$ we have that $\cT_{X,1}\cong D^b(X)$ and this case is well known \cite[Corollary 7.2]{CS18}.
Now take $n>1$ and let us assume that $\cT_{X,n}$ has unique dg-enhancement.

Recall the semiorthogonal decomposition of $\cT_{X,n}$ given in Definition \ref{def:chain-semiorth-decomp}.
Let $\cD_{n+1}$ be a dg-enhancement of $\cT_{X,n+1}=\lin D_1,\cT_{X,n} \rin$, where $D_1 \cong D^b(X)$.
Take $\cD_1$ and $\cD_{n}$ to be the full  dg-subca- \newline tegories of $\cD_{n+1}$ having the same objects as $D_1$  and $\cT_{X,n}$ respectively.
By induction, we have that $\cD_1$ and $\cD_{n}$ are the unique dg-enhancements of $D^b(X)$ and $\cT_{X,n}$ respectively.
Moreover, by \cite[Proposition 4.10]{KL15} we have the existence of a dg-module $\Phi$ given by 
\begin{eqnarray*}
\Phi\colon \cD_{n}^{\op}\otimes \cD_{1} &\longrightarrow & \CC-\rm{dgm}\\ \nonumber
(E_2,E_1)&\longmapsto & \Hom_{\cD_{n+1}}(i(E_1),j(E_2)[1]) \enspace ,
\end{eqnarray*}
where $i\colon \cD_1\rightarrow \cD$ and $j\colon \cD_{n}\rightarrow \cD_{n+1}$ are the corresponding inclusions, such that $\cD_{1}\times_{\Phi} \cD_{n}\cong \cD_{n+1}$ is the gluing of $D_1$ and $D_n$ along $\Phi$, see \cite[Definition 4.1]{KL15}.

Analogously, if we have another dg-enhancement $\cD'_{n+1}$ of $\cT_{X,n+1},$ there is a $\cD^{\op}_{n}\otimes \cD_{1}$-bimodule $\Phi',$ such that $\cD'_{n+1}\cong \cD_1\times_{\Phi'} \cD_{n}.$
By \cite[Lemma 4.7]{KL15}, it is enough to show that $\Phi$ and $\Phi'$ are quasiisomorphic.
This is indeed the case, since
\[H^i(\Hom_{\cD_{n+1}}(E,F))=\Hom^i_{\cT_{X,n+1}}(E,F) \textnormal{ and } H^i(\Hom_{\cD'_{n+1}}(E,F))=\Hom^i_{\cT_{X,n+1}}(E,F) \enspace ,\]
for all $E,F \in \cT_{X,n+1}.$
As a consequence, we have $\cD'_{n+1}\cong \cD_{n+1}.$
\end{proof}


Consider triangulated categories $\cT$ and $\cT'$ and suppose that we have full triangulated subcategories and embeddings $i \colon \cT_1 \to \cT$, $j \colon \cT_2 \to \cT$ and $i' \colon \cT_1 \to \cT'$, $j'\colon \cT_2 \to \cT'$.
Suppose additionally that we have semiorthogonal decompositions
\[\cT = \langle i(\cT_1) , j(\cT_2) \rangle \quad \text{ and } \quad \cT' = \langle i'(\cT_1),j'(\cT_2) \rangle \enspace .\]
If $\cT$ or $\cT'$ admit a dg-enhancement, we deduce that $\cT_1$ and $\cT_2$ admit dg-enhancements \cite[Proposition 4.10]{KL15}.

\bl \label{lemma_dggluing}
Suppose that $\cT$ and $\cT'$ are triangulated categories admitting dg-enhancements $\cD$ and $\cD'$ respectively and that $\cT_{i}$ has a unique dg-enhancement $\cD_{i}$ for $i=1,2.$
Assume that the gluing functors coincide, that is,
\[i^{!}j = i'^!j' \colon \cT_2 \to \cT_1 \enspace .\]
Then $\cT \cong \cT'$.
\el

\begin{proof}
By \cite[Proposition 4.10]{KL15}, we have that $\cD$ is quasiequivalent to $\cD_{1}\times_{\Phi}\cD_{2},$ where $\cD_{1}\times_{\Phi}\cD_{2}$ is the glued dg-category along the bimodule
\begin{eqnarray*}
\Phi\colon \cD^{\op}_{2}\otimes \cD_{1} &\longrightarrow & \CC-\rm{dgm}\\ \nonumber
(E_2,E_1)&\longmapsto & \Hom_{\cD}(i(E_1),j(E_2)[1]) \enspace .
\end{eqnarray*}
Analogously, we have that $\cD'$ is quasiequivalent to $\cD_{1}\times_{\Phi'}\cD_{2}$ where 
\begin{eqnarray*}
\Phi'\colon \cD'^{\op}_{2}\otimes \cD'_{1} &\longrightarrow & \CC-\rm{dgm}\\ \nonumber
(E_2,E_1)&\longmapsto & \Hom_{\cD}(i'(E_1),j'(E_2)[1]) \enspace . 
\end{eqnarray*}
Moreover, we also have that the homotopy category $[\cD_{1}\times_{\Phi}\cD_{2}]$ is equivalent to $\cT$ and
$[\cD_{1}\times_{\Phi'}\cD_{2}]$ is equivalent to $\cT'.$

\begin{claim} The bimodules $\Phi$ and $\Phi'$ are quasiisomorphic.
\end{claim}
\begin{proof} It suffices to show that 
$H^i(\Hom_{D}(i(E_1),j(E_2)[1]))\cong H^i(\Hom_{D'}(i'(E_1),j'(E_2)[1]))$ for all $E_1\in \cD_1$ and $E_2\in \cD_2.$
Note that
\[H^i(\Hom_{D}(i(E_1),j(E_2)[1]))=\Hom^i_{\cT}(i(E_1),j(E_2)[1])) \enspace .\]
By adjunction, we have that $\Hom^i_{\cT}(i(E_1),j(E_2)[1]))\cong \Hom^i_{\cT_1}(E_1,i^!j(E_2)[1]))$.
In the same way we get that \[H^i(\Hom_{D'}(i'(E_1),j'(E_2)[1]))=\Hom^i_{\cT_1}(E_1,i'^!j'(E_2)[1])) \enspace .\]
As $i^!j=i'^!j',$ we obtain that $\Phi$ and $\Phi'$ are quasiisomorphic. 
\end{proof}
By \cite[Lemma 4.7]{KL15}, we finally conclude that $\cD_{1}\times_{\Phi}\cD_{2}$ is quasiequivalent to $\cD_{1}\times_{\Phi'}\cD_{2}$ and therefore $\cT$ and $\cT'$ are equivalent.
\end{proof}

\subsection{Explicit geometric realisations}

Our main goal in this section is to construct smooth projective varieties $Y_{X,n},$ depending on $X,$  and a fully faithful functors $\cT_{X,n}\hookrightarrow D^b(Y_{X,n}).$
We start by finding geometric realisations for the bounded derived categories of representation of $A_n$-quivers by following the steps of \href{https://arxiv.org/pdf/1503.03174.pdf}{\cite{O15}}.

In \cite{O16}, Orlov gives geometric realisations for so-called geometric noncommutative schemes, this also encapsulates the categories in which we are interested.
However, our purposes require an explicit construction.

\bt \label{theorem_gr_quivers} \cite[Theorem 2.6]{O15}
Let $Q$ be a quiver with n ordered vertices.
Then there exists a smooth projective variety $Y_Q$ and an exact functor $u\colon \Rep_{\CC}(Q)\hookrightarrow \Coh(Y_{Q})$ such that the following conditions hold.
\begin{enumerate}
    \item The induced derived functor $u\colon D^b(\Rep_{\CC}(Q))\longrightarrow D^b(Y_Q)$ is fully faithful.
    \item The variety $Y_Q$ is a tower of projective bundles and has a full exceptional collection.
    \item Simple modules $S_{i}$ go to line bundles $L_{i,n}$ on $Y_{Q}$ under $u$. 
    \item any representation $M$ goes to a vector bundle on $Y_{Q}$.
\end{enumerate}
\et

The following example is important for our purposes.

\brem \label{remark_gr_an}
Let us consider the $A_{n}$ quiver.
After applying Theorem \ref{theorem_gr_quivers} we obtain a tower of projective bundles
\[Y_{n}\xlongrightarrow{\pi_{n}} Y_{n-1}\longrightarrow \dots Y_{1}\xlongrightarrow{\pi_{1}}\PP^1\]
and embeddings $u_n\colon D^b(A_{n})\hookrightarrow D^b(Y_{n})$, such that the simple modules $S_{i}$ are sent to line bundles $L_{n,i}\in \Coh(Y_{n})$ for $i=1,\dots, n.$
Following \cite[Theorem 2.6]{O15}, we have an inductive description $L_{n,i}=\pi^*_{n}(L_{n-1,i-1})$ for $i=2,..,n$ with $L_{n,1}=\cO_{Y_{n}}(-1)$ and $L_{n,n}=\cO_{Y_{n}}.$ 

We write $P_i$ for the projective indecomposable corresponding to $S_i$ for  $i=1, \dots ,n.$
Then we have the semiorthogonal decompositions:
$D^b(A_{n}) = \langle S_1 \,, D^b(A_{n-1}) \rangle = \langle D^b(A_{n-1}),P_1 \rangle,$ where $i\colon D^b(A_{n-1}) \rightarrow D^b(A_{n})$ is the fully faithful functor induced by adjoining a vertex on the left. 
We now consider the triangle $M_{n}\longrightarrow P_1 \longrightarrow S_1$ induced by the semiorthogonal decomposition $D^b(A_{n})= \langle S_1 \,, D^b(A_{n-1}) \rangle$ for $P_1.$
Applying $u_n$ to the triangle above, we have 
\[\pi_{n}^*\widetilde{M}_{n} \longrightarrow \cE_{n} \longrightarrow \cO_{Y_n}(-1) \enspace ,\]
where $M_{n} \in D^b(A_{n-1})$ such that $\pi^*_n(\widetilde{M}_{n}) = u_n(M_n)$ and $\cE_n = u_n(P_1)$.
\erem

Following Orlov's steps we now define a smooth projective variety $Y_{X,n},$ depending on $X,$ and a fully faithful functor $v_{n,i}\colon D^b(X)\hookrightarrow D^b(Y_{X,n})$ such that the triangulated category generated by the images of the $v_{n,i}$ has a semiorthogonal decomposition with components $v_{n,i}(D^b(X))$ and the gluing functors coincide with the one of $\cT_{X,n}.$
This allows us to apply Lemma \ref{lemma_dggluing} and obtain the result.

\bd
We define $Y_{X,n}\coloneqq Y_{n}\times X$ and $\widetilde{L}_{n,i}\coloneqq q_{n}^*L_{n,i}$ where $Y_{n}$ and $L_{n,i}$ are as in Remark \ref{remark_gr_an} and $q_n : Y_{X,n} \to Y_n$ is the projection.
\ed

There is the following commutative diagram:
\begin{equation}\label{eq:realisation-quivers}
\xymatrix{\eta^n \colon Y_{X,n} \ar[r]^{\eta_n} \ar[d]^{q_{n}} & Y_{X,n-1} \ar[r]^{\eta_{n-1}} \ar[d]^{q_{n-1}} & \cdots \ar[r] \ar[d]& Y_{X,1} \ar[r]^{\eta_{1}} \ar[d]^{q_1} & X\\
\pi^n\colon Y_{n} \ar[r]^{\pi_{n}}& Y_{n-1}\ar[r]^{\pi_{n-1}} & \cdots \ar[r]& Y_{1}=\PP^1 \enspace.}
\end{equation}
By \cite[Proposition 2.5]{O15}  $Y_{n}=\PP_{Y_{n-1}}(K_{n}),$ therefore we get that $Y_{X,n}=\PP_{Y_{X,n-1}}({q_{n-1}}^*K_{n}),$ where $K_{n}$ fits into the following short exact sequences:
For $n=2,$ we get $K_{2}=\cO_{\PP^1}(-2)\oplus \cO_{\PP^1}(-1)$ and 
\[0\longrightarrow K_{2} \longrightarrow (\cO_{\PP^1}(-1))^{\oplus 4} \longrightarrow \cO_{\PP^1} \longrightarrow 0 \enspace . \]
Subsequently for $n>2,$ we obtain the following short exact sequence in $\Coh(Y_{n-1})$
\[0\longrightarrow K_{n} \longrightarrow (\cO_{Y_{n-1}}(-1) \otimes \pi^*_{n-1}({R}^{-s}))^{\oplus m}\longrightarrow \cE^{\vee}_{n-1} \longrightarrow 0 \enspace ,\]
with $R$ a very ample line bundle on $Y_{n-1}$ and $m,s\in \ZZ_{>0}$.
As a consequence, we obtain a tower of projective bundles.

We then define the following functors
\begin{eqnarray*}  \nonumber
v_{n,i}\colon D^b(X)& \hookrightarrow & D^b(Y_{X,n}) \\ \nonumber
E &\mapsto & {\eta^n}^*(E)\otimes \widetilde{L}_{n,i} \enspace .
\end{eqnarray*}

It also follows from Remark \ref{remark_gr_an} that $\widetilde{L}_{n,i}=\eta^*_{n}(\widetilde{L}_{n-1,i-1})$ for $i=2,\dots,n$ and $\widetilde{L}_{n,1}=\cO_{Y_{X,n}}(-1)$.

\brem
By \cite[Lemma 2.1]{O92}, we have that the functor $\eta^*\colon D^b(X)\longrightarrow D^b(Y_{X,n})$ is fully faithful. 
The functor $v_{n,i}$ is the composition of two fully faithful functors; ${\eta}^*$ and an autoequivalence given by tensoring by a line bundle.
Therefore, the functor $v_{n,i}$ is also fully faithful.
By \cite[Assertion 2.4]{O92} we have that $D^b(X)$ is saturated, therefore $v_{n,i}(D^b(X))$ is an admissible triangulated subcategory.
Moreover, using the adjunction $({\eta^n}^*, \mathbf{R}{\eta_*^n}),$ we obtain that the right adjoint of $v_{n,i}$ is given by
\begin{eqnarray*}  \nonumber
v_{n,i}^!\colon D^b(Y_{X,n})& \hookrightarrow & D^b(X) \\ \nonumber
F &\mapsto & \mathbf{R}{\eta_{*}^n}(F\otimes \widetilde{L}_{n,i}^{\vee}) \enspace .
\end{eqnarray*}
\erem

\bt\label{theorem_gr_X,n}
There is a fully faithful functor $v_{n}\colon \cT_{X,n}\hookrightarrow D^b(Y_{X,n})$.
Moreover, for the standard semiorthogonal decomposition $\cT_{X,n}=\lin D^n_1,\dots,D^n_{n} \rin,$ the image of $\, D^n_i$ under $v_{n}$ is given by $v_{n,i}(D^b(X))$ and the following diagram commutes 
\begin{equation}
\xymatrix{ D^{n-1}_{i-1} \ar[r]^{\eta^*_n} \ar[d]^{v_{n-1,i-1}} & D^{n}_{i} \ar[d]^{v_{n,i}}\\
  D^b(Y_{X,n-1}) \ar[r]^{\eta^*_n} & D^b(Y_{X,n})  
 }
\end{equation}
for $i=2,\dots,n$.
\et

\begin{proof} 
We proceed by induction on $n.$
For $n=1,$ we have that $\cT_{X,1}=D^b(X)$ and $Y_{X,1}=X\times\PP^1=\PP_{X}(\cO_{X}\oplus \cO_{X}).$
As in the proof of Theorem \ref{theorem_gr_quivers}, it holds that $\widetilde{L_{1,1}}=\cO_{\PP^1}$ and by \cite[Lemma 2.1]{O92} the statement follows.

We now assume the assertion for $\cT_{X,n-1}.$
First note that we have already a fully faithful functor
\[\eta^*_{n}\circ v_{n-1}\colon \cT_{X,n-1} \longrightarrow D^b(Y_{X,n})\enspace .\]
Abusing notation, we identify image of this functor with $\cT_{X,n-1}$ itself. Our strategy is to show that $v_{n,1}(D^b(X))\subseteq \cT_{X,n-1}^{\perp}$ and that the triangulated category 
\[\lin v_{n,1}(D^b(X)),\cT_{X,n-1}\rin \subseteq D^b(Y_{X,n})\]
is equivalent to $\cT_{X,n}.$

Let us consider the standard semiorthogonal decomposition $\cT_{X,n-1}=\lin D^{n-1}_1,\dots, D^{n-1}_{n-1}\rin.$ 
In order to prove that $v_{n,1}(D^b(X))\subseteq \cT_{X,n-1}^{\perp},$ it is enough to show that 
\[\Hom_{D^b(Y_{X,n})}({\eta^n}^*(E)\otimes \widetilde{L}_{n,j},{\eta^n}^*(F)\otimes \widetilde{L}_{n,1})=0\] for all $j=2,\dots, n$ and $E, F\in D^b(X).$

Note that $\widetilde{L}_{n,1}=\cO_{Y_{X,n}}(-1)$. Applying the adjunction $({\eta_n}^*,\bR {\eta_n}_*)$ and the projection formula we have
\begin{eqnarray*}
&&\Hom_{D^b(Y_{X,n})}({\eta^n}^*(E)\otimes \widetilde{L}_{n,j},{\eta^n}^*(F)\otimes \cO_{Y_{X,n}}(-1)) \\
&=& \Hom_{D^b(Y_{X,n})}({\eta_n^*}({\eta^{n-1}}^*(E)\otimes(\widetilde{L}_{n-1,j-1})),{\eta^n}^*(F)\otimes \cO_{Y_{X,n}}(-1))\\
&=& \Hom_{D^b(Y_{X,n})}({\eta_n^*}({\eta^{n-1}}^*(E)\otimes(\widetilde{L}_{n-1,j-1})),{\eta_n^*}{\eta^{n-1}}^*(F)\otimes \cO_{Y_{X,n}}(-1))\\
&=& \Hom_{D^b(Y_{X,n-1})}({\eta^{n-1}}^*(E)\otimes\widetilde{L}_{n-1,j-1},{\eta^{n-1}}^*(F)\otimes \bR{\eta_n}_*(\cO_{Y_{X,n}}(-1)))\\
&=&0 \enspace .
\end{eqnarray*} 
The last equality following from the fact that $\bR{\eta_{n*}}(\cO_{Y_{X,n}}(-1))=0$, since $\eta_{n*}$ is a projective bundle.

Consider the two triangulated categories
\[\cT_{X,n}=\lin D_1^n,\cT_{X,n-1}\rin \textnormal{ and } \lin v_{n,1}(D^b(X)), \cT_{X,n-1} \rin\subseteq D^b(Y_{X,n})\]
and note that $\lin v_{n,1}(D^b(X)), \cT_{X,n-1} \rin$, as a full triangulated subcategory of $D^b(Y_{X,n})$, has a dg-enha- \newline ncement.
By Lemma \ref{lemma_uniquedgn} both $\cT_{X,n-1}$ and $D^b(X)$ have unique dg-enhancements.
Moreover, Lemma \ref{lemma_dggluing} implies that it is now enough to show that the gluing functor $v_{n,1}^!v_{n,2}\colon D^b(X)\longrightarrow D^b(X)$ between $D_1$ and $D_2$ in $D^b(Y_{X,n})$ is given precisely by $v_{n,1}^!v_{n,2}(E)=E[-1].$

Indeed, this is the case. Consider:
\begin{eqnarray*}
v_{n,1}^!v_{n,2}(E) & = & v_{n,1}^!({\eta^n}^*(E) \otimes \widetilde{L}_{n,2}) \\
 & = & \bR {\eta_*^n}( {\eta^n}^*(E)\otimes \widetilde{L}_{n,2} \otimes {\widetilde{L}}^{\vee}_{n,1}) \\
 & = & E\otimes \bR \eta_*^{n-1}\bR {\eta_n}_* (\widetilde{L}_{n,2}\otimes \cO_{Y_{X,n}}(1)) \\
 & = & E\otimes \bR {\eta_*^{n-1}}\bR {\eta_n}_*(\eta_n^*(\widetilde{L}_{n-1,1})\otimes \cO_{Y_{n}}(1)) \\
 & = & E\otimes \bR {\eta_*^{n-1}}(\widetilde{L}_{n-1,1}\otimes \bR{\eta_n}_*( \cO_{Y_{n}}(1))) \\
 & = & E\otimes \bR {\eta_*^{n-1}}(q_{n-1}^*(L_{n-1,1}\otimes K_{n})) \enspace .
\end{eqnarray*}
For $n=2,$ we have that $L_{1,1}=\cO_{\PP^1},$ therefore
\begin{eqnarray*}
E\otimes \bR {\eta_*^{1}}(q_{1}^*(L_{1,1}\otimes K_{1}))&=&E\otimes\bR {\eta_*^{1}}(q_{1}^*(K_{2}))\\
&=&E\otimes\bR {\eta_*^{1}}(\cO_{Y_{X,2}}(-2)\oplus \cO_{Y_{X,2}}(-1)^{\oplus 2})=E\otimes \cO_{X}[-1]=E[-1] \enspace .
\end{eqnarray*}
For $n>2,$ we have that
\begin{eqnarray*}
E\otimes \bR {\eta_*^{n-1}}(q_{n-1}^*(L_{n-1,1}\otimes K_{n})) 
&=& E\otimes \bR {\eta_*^{n-2}}(q_{n-2}^*\bR {\pi_{n-1}}_*(
\cO_{Y_{n-1}}(-1)\otimes K_{n})) \enspace ,
\end{eqnarray*}
by (\ref{eq:realisation-quivers}) and the fact that $L_{n-1,1}=\cO_{Y_{n-1}}(-1).$

Using the exact sequences from Remark \ref{remark_gr_an} and since $\rk(K_{n})>2$, we obtain two short exact sequences
\[0\longrightarrow \cO_{Y_{n-1}}\longrightarrow \cE^{\vee}_{n-1}\otimes \cO_{Y_{n-1}}(-1)\longrightarrow  \pi^*_{n}(\cE_{n-2}^{\vee}) \otimes \cO_{Y_{n-1}}(-1) \longrightarrow 0\enspace .\]
\[0\longrightarrow K_{n}\otimes \cO_{Y_{n-1}}(-1) \longrightarrow (\cO_{Y_{n-1}}(-2) \otimes \pi^*_{n}(R^{s}))^{m'}\longrightarrow \cE^{\vee}_{n-1}\otimes\cO_{Y_{n-1}}(-1) \longrightarrow 0 \enspace .\]
As $\bR{\pi_{n-1}}_*(\cO_{Y_{n-1}}(-1))=0,$
we get that $\bR{\pi_{n-1}}_*(\cE_{n-1}^{\vee}\otimes \cO_{Y_{n-1}}(-1))\cong \cO_{Y_{n-2}}.$
Moreover, as $\rk(K_{n-1})>2,$ $\bR{\pi_{n-1}}_*((\cO_{Y_{n-1}}(-2) \otimes \pi^*_{n-1}(R^{\otimes s}))^{m'})=0,$ and therefore 
\[\bR{\pi_{n-1}}_*(K_{n}\otimes \cO_{Y_{n-1}}(-1))=\bR{\pi_{n-1}}_*(\cE_{n-1}^{\vee}\otimes \cO_{Y_{n-1}}(-1))[-1]=\cO_{Y_{n-2}}[-1] \enspace.\]
As a consequence, 
\[v_{n,1}^!v_{n,2}(E) = E\otimes \bR {\eta_*^{n-2}}(q_{n-2}^* \bR {\pi_{n-1}}_*(
\cO_{Y_{n-1}}(-1)\otimes K_{n}))=E\otimes \bR {\eta_*^{n-2}}(\cO_{Y_{X,n-2}})[-1]=E[-1] \enspace .\]
\end{proof}

We conclude this section with two technical lemmas which will require.

\bl \label{lem_sod_orth}
There is a semiorthogonal decomposition of $v_{n}(\cT_{X,n})^{\perp}=\lin B_{1},\dots, B_{m}\rin\subseteq D^b(Y_{X,n}),$ such that $ B_{i}\cong D^b(X)$ for $i=1,\dots, m$ for some $m\in \NN$.
\el

\begin{proof}
The proof goes by induction on $n$.
Let $n=1$ and so $\cT_{X,1}\cong D^b(X)$ and $Y_{X,1}=X\times \PP^1=\PP_{X}(\cO_{X}\oplus \cO_{X}).$
Consider the embedding $w_{1}\colon D^b(X) \longrightarrow D^b(Y_{X,1})$ given by $E \mapsto {\eta^1}^*(E)\otimes \cO_{Y_{X,1}}(-1).$
By \cite[Theorem 2.6]{O92}, there is a semiorthogonal decomposition $D^b(Y_{X,1})=\lin w_{1}(D^b(X)),v_{1}(D^b(X)\rin$ and consequently $v_{1}(\cT_{X,1})^{\perp}\cong w_{1}(D^b(X))$.

Now assume the statement for $n-1.$
As $Y_{X,n}$ is a projective bundle over $Y_{X,n-1},$ by \cite[Theorem 2.6]{O92} there is $r\in \NN$ such that $D^b(Y_{X,n})=\lin C_{-r},\dots, C_{0}\rin$ where $C_{i}\cong D^b(Y_{X,n-1}).$
Moreover, the category $C_{0}$ is the image under the functor $\eta^*_{n}\colon D^b(Y_{X,n-1})\rightarrow D^b(Y_{X,n})$ and $C_{-1}$ is the image under the functor $\eta^*_{n}(-)\otimes \cO_{Y_{X,n}}(-1).$ 
Also note that $v_{n}(\cT_{X,n})=\lin D_{{n-1}_{-1}}^{n-1} , (\cT_{X,n-1})_{0} \rin,$ where  $(\cT_{X,n-1})_{0}$ is the image of the embedding $\cT_{X,n-1} \hookrightarrow C_0$ and  $\cT_{X,n-1}=\lin D^{n-1}_1,\dots ,  D^{n-1}_{n-1} \rin$ is the standard semiorthogonal decomposition.

We also have that $D^b(Y_{X,n-1})=\lin v_{n-1}(\cT_{X,n-1})^{\perp},v_{n-1}(\cT_{X,n-1})\rin$ and hence by the induction hypothesis it holds that
$\cT_{X,n-1}^{\perp}=\lin B_1,\dots, B_s\rin $ with $B_i\cong D^b(X)$ for all $i=1,\dots, s.$ 
As a consequence, we have that
$C_j=\lin B_{1_{j}},\dots, B_{s_{j}}, D^{n-1}_{1_j},\dots D^{n-1}_{{n-1}_j} \rin.$  Therefore we have that $D^b(Y_{X,n})$ splits up into:
\[\lin B_{1_{-r}},\dots, B_{s_{-r}}, D^{n-1}_{1_{-r}},\dots D^{n-1}_{{n-1}_{-r}},\dots, B_{1_{-1}},\dots, B_{s_{-1}}, D^{n-1}_{1_{-1}} ,\dots , D^{n-1}_{{n-1}_{-1}},  B_{1_{0}},\dots, B_{s_{0}}, D^{n-1}_{1_{0}},\dots , D^{n-1}_{{n-1}_{0}}\rin \, \enspace .\]
where $D^{n-1}_{{n-1}_{-1}}$ is precisely the image of $v_{n,1}\colon D^b(X)\rightarrow D^b(Y_{X,n}).$

We can then mutate the semiorthogonal decomposition given above.
Indeed, consider the functor 
\begin{eqnarray*}\nonumber
\mathbb{L}_{D^{n-1}_{{n-1}_{-1}}}\colon  D^b(Y_{X,n}) & \longrightarrow & D^b(Y_{X,n})\\ \nonumber
 F & \longmapsto & C(i_*i^!(F)) \enspace , \nonumber
\end{eqnarray*}
where $i\colon D^{n-1}_{{n-1}_{-1}} \hookrightarrow D^b(Y_{X,n})$ is the inclusion and we mutate as in \cite[Section 2.4]{K09}.

We obtain
\[D^b(Y_{X,n})=\lin B_{1_{-r}},\dots, B_{s_{-r}}, D^{n-1}_{1_{-r}},\dots D^{n-1}_{{n-1}_{-1}},\dots, B_{1_{-1}},\dots,  \mathbb{L}_{D^{n-1}_{{n-1}_{-1}}}(B_{1_{0}}),D^{n-1}_{{n-1}_{-1}},\dots, B_{s_{0}}, D^{n-1}_{1_{0}},\dots D^{n-1}_{{n-1}_{0}}\rin\enspace .\]
In addition, the functor $\mathbb{L}_{D^{n-1}_{{n-1}_{-1}}}$
induces an equivalence $^{\perp}{D^{n-1}_{{n-1}_{-1}}}\mapsto {D^{n-1}_{{n-1}_{-1}}}^{\perp},$ as a consequence \[\mathbb{L}_{D^{n-1}_{{n-1}_{-1}}}(B_{1_{0}})\cong D^b(X)\enspace .\]
By mutating repeatedly as above, we obtain 
\[D^b(Y_{X,n})=\lin B_{1_{-r}},\dots, B_{s_{-r}}, D^{n-1}_{1_{-r}},\dots D^{n-1}_{{n-1}_{-1}},\dots, B_{1_{-1}},\dots,  B_{1_{0}}',\dots, B'_{s_{0}},D^{n-1}_{{n-1}_{-1}}, D^{n-1}_{1_{0}},\dots D^{n-1}_{{n-1}_{0}}\rin\enspace .\]
with $B'_{i_{0}}\cong D^b(X)$.

Since $v_n(\cT_{X,n})=\lin D^{n-1}_{{n-1}_{-1}}, D^{n-1}_{1_{0}},\dots , D^{n-1}_{{n-1}_{0}} \rin $ we have that
\[v_{n}(\cT_{X,n})^{\perp}=\lin B_{1_{-r}},\dots, B_{s_{-r}}, D^{n-1}_{1_{-r}},\dots D^{n-1}_{{n-1}_{-1}},\dots, B_{1_{-1}},\dots,  B_{1_{0}}',\dots, B'_{s_{0}} \rin\]
as desired.
\end{proof}

\brem \label{rem_strongfca}
Let $Y_{X,n}$ as in Theorem \ref{theorem_gr_X,n}, then $\cT_{X,n}$ is an admissible subcategory of $D^b(Y_{X,n}).$ Indeed,  note that $\cT_{X,n}$ is saturated by \cite[Proposition 2.10]{BK89} and therefore $\cT_{X,n}$ is admissible by \cite[Proposition 2.6]{BK89}. 
\erem

\section{Moduli of Bridgeland semistable holomorphic chains}

We begin by introducing a moduli stack analogous to the moduli stack defined and studied by Lieblich \cite{L05}.
These moduli stacks were first studied by Bayer et. al \cite[Section 9]{BLMNPS20} and our notation follows theirs.

Let $Y$ be a smooth projective variety.
From this point on $\cT\subseteq D^b(Y)$ is always an admissible subcategory. 
See Remark \ref{rem_strongfca_general}.

\bd \label{def_Ugluable}
Let $S$ be a $\CC$-scheme. We say that an $S$-perfect (in the sense of \cite[Definition 8.1]{BLMNPS20}) object $E\in D(Y \times S)$ is \emph{universally gluable} if $\Ext^i(E_s,E_s)=0$ for every $i<0$ and $\CC$-point $s \in S$.
We write $D_{\pug}(Y \times S) \subset D(Y \times S)$ for the subcateogry consisting of universally gluable objects in $D(Y \times S)$.
\ed

\brem
Note that by \cite[Lemma 8.3]{BLMNPS20}, we have that $D_{\pug}(Y\times S)\subseteq D^b(Y\times S).$
Moreover, if $S$ is smooth then if $E\in D^b(Y\times S)$ implies that $E$ is $S$-perfect. 
\erem

\bd \label{def_moduli_functor_pug}
We denote
\[\cM_{\pug}(\cT) \colon \Sch/\CC \longrightarrow \Gpds\]
the functor whose value on a $\CC$-scheme $S$ is the set of all $E \in D_{\pug}(Y \times S)$ such that $E_s \in \cT$ for all $s \in S$.
The groupoid structure is given by the standard notion of equivalence: $E \sim E'$ if there exists a line bundle $L \in \Pic(S)$ such that $E \cong E' \otimes q^*L$, where $q \colon Y \times S \rightarrow S$.
\ed

The following is a version of \cite[Theorem 4.2.1]{L05}.

\bp \cite[Proposition 9.2]{BLMNPS20} \label{prop_monster_9.2}
The functor $\cM_{\pug}(\cT)$ is an algebraic stack locally of finite type over $\CC.$
\ep

We now define the moduli stacks which are of our primary interest.

\bd
Let $\sigma \in \Stab_{\Lambda}(\cT)$. We define $\cM^{\beta,\phi}(\sigma)$ to be the substack of $\cM_{\pug}(\cT)$ parameterising the set of $\sigma$-semistable objects of phase $\phi$ and class $\beta$.
In particular
\[\cM^{\beta,\phi}(\sigma)(S) = \left\{ E \in \cM_{\pug}(S) \,\, | \,\, E_s \,\, \sigma\textrm{-semistable, class } \beta\in \Lambda, \,\, \textrm{phase } \phi \right\}\enspace .\]
The groupoid structure is given by the standard notion of equivalence.
\ed

The remainder of this section will be concerned with studying these moduli stacks in the case $\cT = \cT_{X,n}$, where $X$ is a smooth projective variety.
The ultimate aim is prove that the moduli spaces $\cM^{\beta,\phi}(\sigma)$ are algebraic stacks of finite type.
The algebracity will follow from the \emph{open heart property} and \emph{generic flatness} which are addressed in respective sections below.
That the stack is of finite type will follow from \emph{boundedness}.

Let $C$ be a  curve.
For certain stability conditions on $\cT_{C,n}$ we can prove all three of the above properties and in the case of holomorphic triples $\cT_{C}$ for the entire stability manifold.
Furthermore, we can prove some partial results for $\cT_{X,n}$.

\subsection{Boundedness}

As before, let $Y$ be a smooth projective variety and $\cT\subseteq D^b(Y)$ be an admissible subcategory.

\bd A set of objects $B \subseteq \cT$ is
called \emph{bounded} if there is a $\CC$-scheme of finite type $S$, and an object
$E \in \cT_{S}$ such that any object in $B$ is isomorphic to $E_s$ for some $\CC$-point $s \in  S$.
If this rather holds for some $E \in D^b(Y \times S)$ we say that $B$ is bounded in $D^b(Y)$.
\ed

The following lemma states that if boundedness holds for one stability condition on the stability manifold, it must hold for the entire stability manifold.
The result appears in \cite[Theorem 4.2]{PT19}.
We provide a proof for the sake of completeness.

\bl \label{lemma_boundedness_connected} 
Let $\Stab_{\Lambda}^{\circ}(\cT)\subseteq \Stab_{\Lambda}(\cT)$ be a connected component. 
If $\cM^{\beta,\phi}({\sigma})$ is bounded for $\sigma\in \Stab_{\Lambda}^{\circ}(\cT)$ an algebraic stability condition. 
Then, for any $\tau\in \Stab_{\Lambda}^{\circ}(\cT),$ we have that $\cM^{\beta, \phi}(\tau)$ is also bounded for all $\beta\in \Lambda$  and $\phi \in \mathbb{R}.$
\el

\begin{proof}
Let us consider the slicings $\mathcal{P}$ and $\mathcal{Q}$ of $\sigma$ and $\tau$ respectively. 
By connecting $\sigma$ to $\tau$ via a path, we can assume that
\[d(\cP,\cQ)=\rm{inf}\{\delta\in \mathbb{R} \,\, | \,\, \cQ(\phi)\subseteq \cP([\phi-\delta, \phi+\delta]) \textnormal{ for all }\phi \in \mathbb{R}\}=\epsilon <\frac{1}{8}\enspace ,\]
where $d$ is the metric on the stability manifold, see \cite{B07}.
Hence
\[\cQ(\phi)\subseteq \cP((\phi-\epsilon, \phi+\epsilon))\enspace .\]

We now show that $\cM^{\beta,\phi}(\tau)$ is bounded.
Let $E\in \cQ(\phi)$ with $[E]\in \Lambda.$
The semistable factors $F_i$ of $E$, with $1\leq i \leq n_{E}$, satisfy $\phi-\epsilon<\phi_{\sigma}(F_{i})< \phi+\epsilon.$
Since $\sigma$ is algebraic and $\epsilon$ small enough, we have that the map $\cM^{\beta,\phi}({\tau})\rightarrow \NN,$ given by $E\mapsto n(E)$ is bounded. 
As a consequence, the set
\[\{Z(F_i) \,\, | \,\, 1\leq i \leq n_{E}, \textnormal{ with }E\in \cM^{\beta, \phi}(\tau) \}\] 
is finite. 

Let $\phi_{i}\coloneqq\phi(F_i)$ for $i=1,\dots,n.$
Since $\cM^{\beta, \phi_{i}}(\sigma)$ is bounded and applying \cite[Lemma 9.8]{BLMNPS20}, we get that $\cM^{\beta, \phi}(\tau)$ is bounded.
\end{proof}

\bl\label{remark_sem_vectorbundles}
Let $C$ be a  curve, if $E\in \cP_{\sigma_{\alpha}}(\phi)$ for $\sigma_{\alpha}=(Z_{\alpha},\cQ_{C,n})\in\Stab(\cT_{C,n})$ as in Examples \ref{ex_alpha} and $0<\phi<1$.
Then $E_i$ is torsion free for $i=1,\dots,n$.
\el

\begin{proof}
Let $E\in \cQ_{C,n}$ be a $\sigma_{\alpha} $-semistable holomorphic triple.
We can decompose $E_n=T(E_n)\oplus F(E_n),$ where $T(E_n)$ is the torsion part of $E_n$ and $F(E_n)$ the torsion-free part.
Note that ${i_{n}}(T(E_{n}))$ is a subchain of $E$ and $\phi_{\sigma_{\alpha}}({i_{n}}(T(E_{n})))=1$, where $i_n$ is as in Section \ref{sec:CP_rec}.
This contradicts the semistability of $E$.
Therefore $E_{n}$ is torsion free.

Note that the chain 
\[T_{n-1}=0 \longrightarrow \cdots \longrightarrow T(E_{n-1})\longrightarrow 0\]
is a subchain of $E$ with $\phi(T_{n-1})=1$.
Again contradicting semistability and hence $T(E_{n-1})=0$.
Analogously, we prove that $E_{i}$ is torsion free for all $i=1,\dots, n-2$.
\end{proof}

\bl \label{lemma_bounded_alpha}
Let $C$ be curve and let $\sigma_{\alpha}\in \Stab(\cT_{C,n})$, as in Example \ref{ex_alpha} with $\alpha\in\QQ^n$.
Then $\cM^{\beta, \phi}(\sigma_{\alpha})$ is bounded for all $\beta\in \ZZ^{2n}$ and $0<\phi<1$.
\el

\begin{proof}
The moduli space of $\sigma_{\alpha}$-semistable holomorphic chains of vector bundles was constructed in \cite[Theorem 1.6]{S03}.
Therefore, by Lemma \ref{remark_sem_vectorbundles} we obtain that $\cM^{\beta, \phi}(\sigma_{\alpha})$ is bounded.
\end{proof} 

\brem
Under the assumptions of Lemma \ref{lemma_bounded_alpha}, if $E\in \cP_{\sigma_{\alpha}}(1)$, we have that $\sum^{n}_{i=1}r_{i}=0$.
As $r_i\geq0$ for all $i,$ we get that $r_i=0.$
Therefore $E_i$ is a torsion sheaf for each $i$.
After fixing $\beta=(0_{i},d_i)_{i=1,\dots,n}$ we get $\cM^{\beta, 1}(\sigma_{\alpha})\cong \Sym^{n}(C)$ and is bounded.
\erem 

The following corollary holds under the assumption that the support property is satisfied for $\sigma_{\alpha}$.
Let $\Stab^{\circ}(\cT_{C,n})$ be the connected component containing $\sigma_{\alpha}$.

\bc \label{cor_boundedness_chains}
Let $C$ be a curve and $\Stab^{\circ}(\cT_{C,n})$ the connected component containing $\sigma_{\alpha}.$
Then for every $\sigma\in \Stab^{\circ}(\sigma_{\alpha}),$ the set $\cM^{\beta, \phi}(\sigma)$ is bounded for every $\beta\in 
\ZZ^{2n}$ and $\phi\in \RR.$
\ec

\bc \label{cor:boundedness_curve_triples}
For $\sigma\in \Stab(\cT_{C}),$ then $\cM^{\beta, \phi}(\sigma)$ is bounded for all $\beta\in \ZZ^4$ and all $\phi\in \RR.$
\ec
\begin{proof}
It is proven in \cite[Theorem 1.1]{MRRHR20} that $\Stab(\cT_{C})$ is connected.
The statement follows from Lemmas \ref{lemma_bounded_alpha} and \ref{lemma_boundedness_connected}.
\end{proof}

\subsection{The open heart property}

We begin by stating the open heart property.

\bd\label{def:OHC}
Let $S$ be a $\CC$-scheme of finite type and $\cA \subset \cT$ a noetherian heart. 
We say that $\cA$ satisfies the \emph{open heart property} if for every $E \in \cT_S$ and smooth $\CC$-point $s \in S$ with $E_s \in \cA$, there exists an open neighbourhood $s \in U \subset S$ such that $E_U \in \cA_U$.
\ed

In the case of $D^b(X)$ the open heart property was proven by Abramovich and Polishchuk.

\bt \label{prop_OHCX} \cite[Proposition 3.3.2]{AP06}
Let $\cA \subseteq D^b(X)$ be a noetherian heart.
Then $\cA$ satisfies the open heart property.
\et

The rest of the section will be concerned with proving the analogous result in our setting:

\bt \label{thm:OHC}
Any noetherian heart $\cA \subset \cT_{X,n}$ satisfies the open heart property.
\et

\begin{lemma}\label{lemma:heart_geom_cat}
Suppose that $\cT = \langle D_1, D_2 \rangle$ is a semiorthogonal decomposition with $D_1 \cong D^b(X_1)$ with $X_1$ a smooth projective variety.
If $D_2$ admits a noetherian heart, then $\cT$ also admits a noetherian heart.
\end{lemma}
\begin{proof}
Let $\cA \subset D_2$ be a noetherian heart.
We claim that 
\[\cA' = \gl \left(\Coh(X_1)[m] \, , \, \cA\right)\]
for $m>>0$ satisfies the gluing conditions given in Proposition \ref{CP}.
Assuming this claim to be true it would follow from Remark \ref{Rem_filtrationCP} that $\cA'$ is noetherian and hence would conclude the proof.

Let us prove the claim.
It suffices to show
\[\Hom_{\cT}^{\leq0}\left(i(\Coh(X))[m],j(\cA)\right)= 0\enspace ,\]
where $i: D_1 \to \cT$ and $j: D_2 \rightarrow \cT$ are the inclusions.
Consider $E \in \Coh(X)$ and $F \in \cA$, then
\[\Hom_{\cT}\left(i(E),j(F)\right) = \Hom_{D_2}\left(j^*i(E),F\right)\enspace .\]
Consider the composition $j^*i : D_1 \to D_2$.
By \cite[Proposition 2.5]{K08}, there are $a,b\in \ZZ$ such that for all $E\in \Coh(X)$ we have that $j^*i(E)\in D^{[a,b]}$, where $(D^{\leq 0},D^{\geq 0})$ is the t-structure given by $\cA$ on $D_2.$
Now we choose $m>>0$ such that $\Hom^{\leq 0}_{D_2}(j^*i(E)[m],F) = 0$.
\end{proof}

The next step towards proving the open heart property is to prove that the recollement of two hearts is stable under base change.
By this we mean that one can either base change two hearts and then consider the recollement or first take the recollement of two hearts and then base change; the resulting heart is the same.

\bp \textup{[Recollement is stable under base change]} \label{prop:rec_base_change}
Let $j\colon \cT\hookrightarrow D^b(Y)$ be an admissable triangulated subcategory, where $Y$ is a smooth projective variety, with $\cT=\lin D_1,D_2\rin$ and $S$ be a projective variety.
Let $\cA^{1}\subseteq D_1$ and $\cA^{2}\subseteq D_2$ be hearts of bounded t-structures.
Then
\[\rec(\cA^1_S , \cA^2_S) = \rec(\cA^1,\cA^2)_S\enspace .\]
Moreover, if $\cA^1$ and $\cA^2$ satisfy gluing conditions, for $E \in \rec(\cA^1,\cA^2)_S$ we have that ${i_2}^!(E) \in \cA^2,$ where $i_2\colon D_2 \hookrightarrow \cT$ is the inclusion functor.
\ep

\begin{proof}
It suffices to show that $\rec(\cA^1_S , \cA^2_S) \subseteq \rec(\cA^1,\cA^2)_S$, since inclusion of hearts implies equality.
For brevity we write $\phi_n = \ \bR p_*(- \otimes (q^*L)^n)$ where $L$ is an ample line bundle on $S$.

We first note that $F \in \rec(\cA^1,\cA^2)_S$ if and only if $\phi_n F \in \rec(\cA^1,\cA^2)$ for all $n \gg 0$ and that this happens if and only if we have the following two exact triangles:
\begin{equation}\label{eq:rec_exact_triangles}
    F_2 \rightarrow \phi_n F \rightarrow F_1 \rightarrow F_2[1]\quad\quad\quad
    F'_3 \rightarrow \phi_n F \rightarrow F_2 \rightarrow F'_3[1] \enspace,
\end{equation}
for some $F_1 \in \cA^1$, $F_2 \in D_2^{\geq 0}, F'_2 \in D^{\leq 0}_{2},$ where $(D^{\leq 0}_{j},D_j^{\geq 0})$ is the t-structure with heart $\cA^{j}$ for $j=1,2,$ and $F'_3\in {^{\perp}} D_2$.
Note that these triangles are unique.

Suppose now that $E \in \rec(\cA^1_S , \cA^2_S)$.
Thus we have the exact triangle
\[E_2 \rightarrow E \rightarrow E_1 \rightarrow E_2[1] \enspace ,\]
where $E_2 = j^!E \in (D_2^{\leq 0})_S$ and $E_1 \in \cA^1_S$.
Applying the functor $\phi_n$ we get
\[\phi_n E_2 \rightarrow \phi_n E \rightarrow \phi_n E_1 \rightarrow \phi_nE_2[1] \enspace .\]
Choosing $n$ large enough allows us to conclude that $\phi_n E_1 \in \cA^1$ and $\phi_n E_2 \in D_2^{\leq 0}$.
This recovers the first exact sequence from (\ref{eq:rec_exact_triangles}) for $\phi_n E$. 

Similarly, consider the triangle
\[E_3 \rightarrow E \rightarrow E_2' \rightarrow E_3[1]\]
with respect to the decomposition $\cT_{S} = \langle ({D_2})_S,({^{\perp}}{D_2})_S\rangle$.
Again, $E'_2 \in (D_2^{\leq 0})_S$ and hence, for $n \gg 0$, we have that $\phi_n E'_2 \in D_2^{\leq 0}$.

To argue that $\phi_n E_3 \in {^{\perp} D_2}$ we note that by \cite[Proposition 2.1.3]{AP06} we have $E_3 \otimes q^*L \in ({^{\perp}D_2)_S}$ for any $L$.
Given that $S$ is projective,
\[\bR p_*(({^{\perp}D_2})_S) \subset {^{\perp}D_2}\]
by Proposition \ref{prop_base_change_sod}. 
We now assume that $\rec(\cA^1,\cA^2)=\gl(\cA^1,\cA^2)$ and $E\in\rec(\cA^1,\cA^2)_{S},$ in this case, for $n\gg 0$ we get that ${i_2}^!(\phi_{n}(E))=\phi_n{E_2}\in \cA^2$, which is precisely the definition of ${i_2}^!E=E_2 \in \cA_{S}^2$.
\end{proof}

\bc \label{corollary_base_change_gluing}
Let $\cB\subseteq D^b(X)$ be the heart of a bounded t-structure.
If  $\cA=\gl(\cB,\dots,\cB)\subseteq \cT_{X,n},$ then $\cA_{S}=\gl(\cB_{S},\dots,\cB_{S})$.
\ec

\bdem
Note that if we take the same heart in each component the gluing conditions are automatically satisfied.
The result follows from the recursive definition of $\cA,$ Proposition \ref{prop:rec_base_change} and \cite[Proposition 2.8.12]{MR18}.
\edem

\bc\label{cor:rec_intersection_heart}
Suppose that $\cT = \langle D_1,D_2 \rangle$ and $\cA^i \subset D_i$ are  hearts of bounded t-structures such that $\widetilde{\cA} = \rec (\cA^1,\cA^2)$.
Then
\[(\cA^2)_S = (D_2)_S \cap \widetilde{\cA}_S \enspace .\]
\ec

\begin{proof}
It follows directly from the definition of recollement $\cA^2=D_2\cap \rec(\cA^1,\cA^2),$ then the statement follows directly from Proposition \ref{prop:rec_base_change}.
\end{proof}

Let $\cA\subseteq \cT_{X,n}$ be a noetherian heart and consider the realisation $Y_{X,n}$ as in Theorem \ref{theorem_gr_X,n}.
By Lemma \ref{lem_sod_orth}, there is a decomposition $\cT^{\perp}_{X,n}=\lin \cB_1,\dots, \cB_{m} \rin$ with $\cB_{i}\cong D^b(X).$
Recursively applying Lemma \ref{lemma:heart_geom_cat} we construct a noetherian heart
\[\widetilde{\cA}=\gl(\cB,\cA)\subseteq D^b(Y_{X,n})\]
with $\cB = \gl\left( \Coh(X)[n_1], \dots ,\Coh(X)[n_k] \right) \subseteq {\cT_{X,n}}^{\perp}.$

\bc \label{cor_Li_rightexact}
Let $\cA\subseteq \cT_{X,n}$ be the noetherian heart of the t-structure $(D^{\leq 0},D^{\geq 0})$ on $\cT_{X,n}$.
Then the functor $\bL i^*_{s}$ is right t-exact with respect to $(D^{\leq 0}_{S}, D^{\geq 0}_{S})$ and $(D^{\leq 0}, D^{\geq 0})$.
\ec

\bdem
Let $E\in D^{\leq 0}_{S}.$
By the definition in Theorem \ref{thm:BBDheart} we have that $E\in \widetilde{D}^{\leq 0}_{S}$, where $(\widetilde{D}^{\leq 0},\widetilde{D}^{\geq 0})$ is the t-structure with heart $\widetilde{\cA}.$
It follows from \cite[Lemma 2.5.3]{AP06} that $\bL i^*_{s}(\widetilde{D}^{\leq 0}_{S})\subseteq \widetilde{D}^{\leq 0}.$ Note that  $\bL i^*_{s}((\cT_{X,n})_{S})\subseteq \cT_{X,n}$.
Consequently $\bL i^*_{s}(E)\in \widetilde{D}^{\leq 0}\cap \cT_{X,n}=D^{\leq 0}.$
\edem

We also get a result analogous to \cite[Lemma 2.6.2]{AP06}. 

\bc\label{cor_Lem2.6.2AP}
Let $E\in \cA_{S},$ then there is $H\in \cA$ and $n\in \ZZ$ such that $p^*(H)\otimes q^*(L)^n\twoheadrightarrow E$ in $\cA_{S}.$
\ec

\bdem
By the definition of recollement $E\in \widetilde{\cA},$ after applying \cite[Lemma 2.6.2]{AP06} there is $G\in \widetilde{\cA}$ and $n\in \ZZ$ such that $p^*(G)\otimes L^n\twoheadrightarrow E$ in $\widetilde{\cA}_{S}.$
By applying $j^!,$ we obtain $p^*(j^!(G))\otimes q^*(L)^n\twoheadrightarrow E$ in $\cA_{S}$. 
Indeed, it holds that $j^!(\widetilde{\cA})\subseteq \cA$ and $j^!(p^*(G)\otimes L^n)=p^*(j^!(G))\otimes q^*(L)^n$.
Therefore, take $H=j^!(G)\in \cA$.
\edem 

Now we are in a position to prove Theorem \ref{thm:OHC}.

\begin{proof}[Proof of Theorem \ref{thm:OHC}] 
Let $E\in (\cT_{X,n})_{S},$ such that $E_{s}\in \cA\subseteq \widetilde{\cA}.$
Theorem \ref{prop_OHCX} ensures the open heart property is satisfied for the heart $\widetilde{\cA}_{S}.$
Thus there is an open set $s\in U\subseteq S$ such that $E_{U}\in \widetilde{\cA}_{U}$ and by definition $E_{U}\in (\cT_{X,n})_{U}.$
Since hearts constructed via gluing coincide with those constructed via recollement as in Remark \ref{rem_glu_rec}, by Corollary \ref{cor:rec_intersection_heart} we then have that $E_{U}\in \cA_{U}.$
\end{proof}

\brem \label{rmk:OHC-general}
Using the techniques above, one can prove the open heart property for noetherian hearts for a wider class of triangulated categories $\cT$.
That is, where $\cT$ admits a realisation $D^b(Y)$ and $\cT^{\perp}$ is \emph{geometric} i.e.\ admit semiorthogonal decompositions with semiorthogonal components given by $D^b(X_i)$ for various smooth projective $X_i$.
\erem

\subsection{Generic flatness}
Let $Y$ be a smooth projective variety and $\cT\subseteq D^b(Y)$ an admissible subcategory.

\bd \label{def_t-flat}
Let $\cA\subseteq \cT$ be the heart of a bounded t-structure.
We say that $E\in \cA_{S}$ is \emph{t-flat} if for every $s\in S,$ we have that $E_{s}\in \cA.$
\ed

\bd \label{def:GF}
Let $\cA\subseteq \cT$ be a heart of a bounded t-structure, we say that $\cA$ satisfies the \emph{generic flatness property} if for all $\cE\in \cA_{S},$ where $S$ is a projective variety, there is an open set $U\subseteq S$ such that for all $s\in U,$ we have that $\cE_{s}\in \cA.$  
\ed

As in \cite[Proposition 3.5.3]{AP06}, using the techniques of the proof of the open heart property, we can immediately provide a partial result of generic flatness of a heart $\cA \subset \cT_{X,n}$.

\bp
For $E\in \cA_{S},$ then there is a dense set $Z\subseteq S$ such that $E_{s}\in \cA$ for every $s\in Z.$ 
\ep

\bdem
We consider $\cA\subseteq \widetilde{\cA}$ as in the proof of Theorem \ref{thm:OHC}.
Let $E\in \cA_{S}\subseteq \widetilde{\cA}_{S}$.
By \cite[Proposition 3.5.3]{AP06}, there is a dense set $Z\subseteq S,$ such that for all $s\in Z$ we have that  $E_{s}\in \widetilde{\cA}.$
Due to the fact that $E_{s}\in \cT_{X,n}$ and $\cA=\cT_{X,n}\cap \widetilde{\cA}$, it follows that $E_{s}\in \cA$ for all $s\in Z$.
\edem

We start by proving the generic flatness property in the case of curves.
The following lemma follows from the same arguments of \cite[Lemma 4.7]{T08}, after replacing the $K3$ surface with a curve.

\bl \label{lemma_gen_flat_curves}
Let $C$ be a curve and $\sigma=(Z,\cA)\in \Stab(C)$.
Then $\cA$ satisfies generic flatness.
\el

\bdem
First note that by \cite[Theorem 2.7]{M07}, we have  $\cA=\mathcal{P}_{\sigma_\mu}((r,1+r])$ where $r=m+\theta$ with $\theta\in [0,1)$ and $m\in \ZZ,$ where  $\sigma_{\mu}$ is given by slope stability and $\mathcal{T}_{\theta}=\mathcal{P}_{\sigma_{\mu}}(\theta, 1]$ and $\mathcal{F}_{\theta}=\mathcal{P}_{\sigma_{\mu}}(0, \theta].$
Therefore, we get $\cA=\lin \mathcal{F}_{\theta}[m+1],\cT_{\theta}[m]\rin$.
It is enough to prove the statement for $m=0$.
Let $\mathcal{E}\in \cA_{S},$ we have that $\bR p_*(\cE\otimes q^*(L)^{n})\in \cA$ for $n>>0$.
Note that the cohomology of $\cE$ is concentrated in degree $-1,0$.

The spectral sequence $$E^{i,j}_{2}=\bR^{i} p_*(H^i(\mathcal{E})\otimes  q^*(L)^{n})\Rightarrow \bR^{i+j}(\cE\otimes q^*(L)^{n})\in \cA$$ degenerates for $n>>0$, from which it follows that $H^i(\cE)=0$ unless $i=-1,0$.
Then by \cite[Theorem 2.3.2]{HL10}, there is an open set $U\subseteq S$  and a filtration $$0=F^0\subseteq F^1\subseteq \cdots F^l=H^{-1}(\cE)_{U}$$ such that $F^{i+1}/F^{i}$ are $U$-flat for $i=1,\dots, l$.
Moreover, for $s\in S$, the filtration
\[0=F_s^0\subseteq F_s^1\subseteq \cdots \subseteq F_s^l=H^{-1}(\cE)_{s}\]
is precisely the HN-filtration of $H^{-1}(\cE)_{s}$ with respect to $\mu$-stability.
By \cite[Proposition 3.5.3]{AP06}, we have that there is dense set $S'\subseteq S$, such that $\cE_{s}\in \cA$ for $s\in S'$.
This implies that for every $s\in S'$ we get that $H^{-1}(\cE)_{s}\in \cF_{\theta}$.
Since $F^{i+1}/F^{i}$ are $U$-flat, we have that $[F_{s}^{i+1}/F_{s}^{i}]=[F_{s'}^{i+1}/F_{s'}^{i}]$ for $s\in U$ and $s'\in S'$ and that $\mu(F_{s}^{i+1}/F_{s}^{i})=\mu(F_{s'}^{i+1}/F_{s'}^{i})\leq -\cot(\pi\theta)$.
It implies that for all $s\in U,$ we get that $H^{-1}(\cE)_{s}\in \cF_{\theta}$.
Analogously for $H^0(\cE)$.
\edem 

\bl \label{Lemma_GFforrec}
Let $\cT$ be a triangulated category and $\cT=\lin D_1,D_2\rin$ a semiorthogonal decomposition.
Suppose there are hearts $\cA^j\subseteq D_j$ for $j=1,2$, satisfying generic flatness, then the heart $\cA=\rec(\cA^1,\cA^2)$ also satisfies generic flatness. 
\el

\bdem
Let $E\in \cA_{S}$.
By Proposition \ref{prop:rec_base_change}, we have that $\cA_{S}=\rec(\cA_{S}^1,\cA_{S}^2)$ and that ${i}_1^*(E)\in \cA_{S}^1,$ ${i}_2^!(E)\in {D^{\geq 0}_2}_{S}$ and ${i}_2^{*}(E)\in D^{\leq 0}_2,$ where $(D^{\leq 0}_2,D^{\geq 0}_2)$ is the t-structure with heart $\cA^2$.
From Lemma \ref{lemma_gen_flat_curves}, it follows that there are open sets $U_1$ and $U_2,$ such that   for all $s\in U_1 \textnormal{ we have } {i}_1^*(E)_{s}\in \cA^1$ and $\textnormal{ for all } s\in U_2 \textnormal{ we have } {i}_2^!(E)_{s}\in {D_2}^{\geq 0}$ and ${i}_2^*(E)\in {D_2}^{\leq 0}.$

Finally we define $U\coloneqq U_1\cap U_2$ and note that it follows directly from the definition of recollement that $E_{s}\in \cA$ for all $s\in U$.
\edem

\bc \label{corollary_GFgluing}
Let $\cA$ be a gluing heart with respect to the standard semiorthogonal decomposition $\cT_{C,n}=\lin D_1,\dots, D_n\rin $.
Then $\cA$ satisfies the generic flatness property.
\ec

\bdem
Since $\cA$ is a gluing heart, there exist hearts $\cA_j\subseteq D_j \cong D^b(C)$ for $j=1,\dots,n$ such that $\cA=\gl(\cA_1,\dots,\cA_n)\subseteq \cT_{C,n}$.
By Lemma \ref{lemma_gen_flat_curves}, we have that $\cA_j$ satisfies the generic flatness property.
Moreover, it follows from Remark \ref{rem_glu_rec} that if $\cT=\lin D_1,D_2\rin$ and two hearts $\cB^1\subseteq D_1 $ and $\cB^2\subseteq D_2$ satisfy gluing conditions, then $\gl(\cB^1,\cB^2)=\rec(\cB^1,\cB^2)$.
By the recursive construction of $\cA=\gl(\cA^1,\dots, \cA^n)$ and Lemma \ref{Lemma_GFforrec}, we obtain that $\cA$ satisfies the generic flatness property. \edem 

We will need the following result which follows from the same arguments given in \cite[Lemma 3.15]{T08}.

\bl \label{lem_3.15}
Let $\sigma=(Z,\cA)\in \Stab(\cT_{C,n})$ be an algebraic stability condition.
Assume that $\cM^{\beta,\phi}(\sigma)$ is bounded for all $\phi\in \RR$ and $\beta\in \mathbb{Z}^{2n}$.
Then for $\phi\in (0,1)$ and $G\in \cA$ the following set of 
\[Q(G,\phi) = \{ E\in \cA \,\, | \,\, \textnormal{there exists a surjection }  G\twoheadrightarrow E \in \cA \textnormal{ and } \phi(E) \leq \phi \} \]
is bounded in $D^b(Y_{C,n})$.
\el

We now assume that for an algebraic stability condition $\sigma=(Z,\cA)\in \Stab(\cT_{C,n})$ the set $\cM^{\beta,\phi}(\sigma)$ is bounded for all $\phi\in \mathbb{R}$ and $\beta\in \mathbb{Z}^{2n}$ and that $\cA$ satisfies generic flatness.
Let $E\in \cA_{S}$ be t-flat and take $\phi\in (0,1)$.
We consider the following functors
\[\Quot(E,\phi), \,\,\,\, \left(\Sub(E,\phi)\right)\colon (\Sch/S)\rightarrow \Sets\]
defined as follows:
A scheme $T$ over $S$ is mapped to pairs of the form $(F,E_{T}\rightarrow F)$ (respectively $(F,F\rightarrow E_{T})$) where $F\in\cM_{\pug}({\cT_{C,n}})(T)$ such that:
\begin{enumerate}
\item For each $t\in T,$ we have that $F_{t}\in \cA$ and $\phi(F_{t})\leq \phi$ (respectively $\phi(F_{t})\geq \phi$).
\item For each closed point $t\in T,$ the induced morphism $E_t\rightarrow F_t$ is surjective (respectively $F_t \rightarrow E_t$ injective) in $\cA$.
\end{enumerate}

\begin{remark}
These functors are a subspaces of the quot spaces defined in \cite[Definition 11.3]{BLMNPS20}.
It follows from \cite[Proposition 11.6]{BLMNPS20} that they are algebraic spaces.
\end{remark}

We prove the following proposition  \cite[Proposition 3.17]{T08}.
The proofs are essentially the same with some minor modifications.
In particular, we incorporate the techniques of \cite{BLMNPS20}.

\bp\label{prop_toda_3_17}
For any $\phi\in (0,1)$ there exist $S$-schemes $\cQ(E,\phi),\cS(E,\phi)$, of finite type over $S$, and $S$-morphisms
\begin{eqnarray*}
\cQ(E,\phi)&\rightarrow&\Quot(E,\phi) \enspace ,\\
\cS(E,\phi)&\rightarrow &\Sub(E,\phi)\enspace ,  
\end{eqnarray*}
which are surjective on $\CC$-valued points of $\Quot(E,\phi)$ and $\Sub(E,\phi)$.
\ep

\bdem
Let $E\in \cA_{S}$.
By Corollary \ref{cor_Lem2.6.2AP} we have that there is an object $H\in\cA$, some integer $n \in \ZZ$ and a surjection ${H}_S \otimes L^{-n}\twoheadrightarrow E$ in $\cA_{S}$, for $i=1,\dots,n$. 
Corollary \ref{cor_Li_rightexact} states that the functor $\bL i_*\colon (\cT_{C,n})_{S} \rightarrow \cT_{C,n}$ is right t-exact with respect to $\cA$ and therefore the kernel of $H\twoheadrightarrow E$ is also t-flat.
As a consequence, we obtain a morphism from $H_{s}\twoheadrightarrow E_{s} $ for each $s\in S$.

By Lemma \ref{lem_3.15} there is a $\CC$-scheme of finite type $Q$ and $F\in D^b(Y_{C,n}\times S)$, such that any object in $Q(H,\phi)$ is isomorphic to $F_{q}$ for some $q \in Q.$
As $\cA$ satisfies both generic flatness and the open heart property by Theorem \ref{thm:OHC}, the set $Q_1=\{q \in Q \,\, | \,\, F_{q}\in \cA\}$ is open.

Define $Q_2=Q_1\times S$.
By \cite[Lemma 8.9]{BLMNPS20} there is an open $U\subseteq Q_2$ such that 
\begin{eqnarray*}
\underline{\Hom}_{U}(E,F)\colon  (\Sch/U)^{\rm{op}} & \longrightarrow & (\Sets) \\
T &\longmapsto & \Hom_{D(X_{T})}(E_T,F_T)
\end{eqnarray*}
is representable by an affine scheme $Z_{U}$.

Once again there is an open set $V \subseteq Q_2 \setminus U$, such that the functor above over $V$ is representable by $Z_{V}$.
Recursively we construct a scheme $Q_3$ whose $\CC$-points are in bijection with the $\CC$-points of $Q_2$ and where the functor above over $Q_3$ is representable by $Z$, an affine scheme of finite presentation over $Q_3$.
Moreover, note that since $Q_2$, and therefore also $Q_3$, is of finite type over $S$, we have that $Z$ is of finite type over $S$.

Let us consider the universal family $E_{Z}\rightarrow F_{Z}$ and the triangle $K\rightarrow E_{Z}\rightarrow F_{Z}$ in $D^b(Y_{C,n}\times Z)$.
For $q\in Z,$ we have that $F_{q}\in\cA$.
As a consequence, the morphism $E_{q}\rightarrow F_{q}$ is surjective in $\cA$ if and only if $K_{q}\in \cA$.
Then, we define
\[\cQ(E,\phi)\coloneqq \{q\in Z \,\,| \,\, K_{q}\in \cA\}\]
which induces a morphism to $\Quot(E,\phi)$ that is surjective on $\CC$-valued points.
We get that $\cQ(E,\phi)$ is an open subscheme of $Z$.
Indeed, we apply the open heart property and generic flatness of $\cA$.

The arguments for $\Sub(E,\phi)$ are the same as in \cite[Proposition 3.17]{T08}.
\edem

\bp\label{prop_GFuptoaction} \cite[Proposition 3.18]{T08} 
Let $\sigma=(Z,\cA)\in \Stab(\cT_{C,n})$ be as in Proposition \ref{prop_toda_3_17}. If there is $g \in \GL$ such that $\sigma'=(Z',\cA')=\sigma \cdot g$ and $\sigma'$ is algebraic then $\cA'$ also satisfies the generic flatness property.
\ep

\bc
Let $\sigma=(Z,\cA)\in \Stab(\cT_{C})$ an algebraic stability condition, then $\cA$ satisfies the generic flatness property. 
\ec

\bdem
By Theorem \ref{them_classificationC_connected}, we have that either $\sigma$ is a gluing stability condition or $\sigma$ satisfies $(2)$ or $(3)$.
If $\sigma$ is a gluing stability condition, our statement follows from Corollary \ref{corollary_GFgluing}.

If $\sigma$ satisfies $(2)$, then there is a $g\in \GL$ such that $\sigma \cdot g$ is a gluing stability condition, in this case the result then follows from Proposition \ref{prop_GFuptoaction}.

If $\sigma$ satisfies $(3)$, then there is $g\in G$, such that $\sigma \cdot g=(Z',\cA')$ has the following property:
There is a non-gluing stability condition $\sigma''=(Z'',\cA'')$ satisfying $(2)$ with such that $\cA'=\cA''$. 
Since generic flatness is a property of the heart, we apply the same argument as before.
\edem

\subsection{Algebraic moduli stacks}

Following the strategy laid out in \cite{AP06}, we prove, using the open heart property, that generic flatness is sufficient for the algebraicity of the substack of Bridgeland semistable chains.

First we state the following lemma.
The proof is exactly the same as in \cite{T08}.

\bl \cite[Lemma 3.13]{T08}\label{lem_3.13T}
Let $\sigma=(Z,\cA)\in \Stab(\cT_C)$ be an algebraic stability condition.
For a smooth quasi-projective variety $S$ and $\cE \in \mathcal{M}_{\pug}(\cT_{C})(S)$, assume that the locus
\[S^{\circ} = \left\{s \in S \, | \, \cE_s \textnormal{ is of numerical type } \beta\in \ZZ^4 \textnormal{ and } \cE_s \in \mathcal{P}_{\sigma}(\phi)\right\}\]
for $\phi \in \mathbb{R}$ is not empty.
Then there is an open subset $U \subseteq S$ which is contained in $S^{\circ}.$
\el


We now state and prove the main result.

\bt\label{thm_main1}
Let $\sigma\in \Stab(\cT_C)$ be a stability condition, then $\mathcal{M}^{\beta,\phi}(\sigma)$ is an algebraic stack of finite type over $\CC$ for all $\beta\in \ZZ^4$ and $\phi\in \RR.$
\et

\bdem 
First assume that $\sigma$ is an algebraic stability condition.
Following standard arguments (see for example \cite[Lemma 3.6]{T08}), Lemma \ref{lem_3.13T} implies that $\mathcal{M}^{\beta,\phi}(\sigma)$ is an open substack of $\cM_{\pug}(\cT_{C})$.

By Corollary \ref{cor:boundedness_curve_triples}, we obtain that $\mathcal{M}^{\beta,\phi}(\sigma)$ is bounded.
Therefore $\mathcal{M}^{\beta,\phi}(\sigma)$ is an algebraic stack of finite type over $\mathbb{C}$.
See \cite[Lemma 9.7]{BLMNPS20}.

The result for a non-algebraic stability condition follows from the algebraic case proved above.
We omit the proof since it is exactly the same as in \cite[Proposition 3.20, Step 3]{T08}  which relies on the well-behaved wall and chamber decomposition, see \cite[Proposition 3.3]{BM11}.
\edem 

\bc
For every $\sigma=(Z,\cA)\in \Stab(\cT_C)$, we have that $\cA$ satisfies generic flatness. 
\ec

\bdem
An adapted version of the arguments of \cite[Proposition 4.12]{PT19} to our set up follows from Theorem \ref{thm_main1}, Corollary \ref{cor_Li_rightexact} and Corollary \ref{cor_Lem2.6.2AP}.
\edem

Under the assumption that $\sigma_{\alpha} \in \Stab^{\circ}(\cT_{C,n})$ satisfies the support property we obtain the following proposition.
Let $\Stab^{\circ}(\cT_{C,n})$ be the connected of $\sigma_{\alpha}.$

\bp \label{prop_algstack_chains} 
Let $\sigma=(Z,\cA)\in \Stab^{\circ}(\cT_{C,n})$ be a gluing algebraic stability condition.
Then $\cM^{\beta,\phi}(\sigma)$ is an algebraic stack of finite type over $\CC$ for all $\beta\in \ZZ^{2n}$ and $\phi\in \RR$.
\ep

\bdem
By Corollary \ref{cor_boundedness_chains}, we have that $\cM^{\beta,\phi}(\sigma)$ is bounded. 
Subsequently, $\cA$ satisfies the open heart property (Theorem \ref{thm:OHC}).
Moreover, Corollary \ref{corollary_GFgluing} tells us that $\cA$ satisfies the generic flatness property.
By Proposition \ref{prop_GFuptoaction}, for every $g \in \GL$, we also get that if $\sigma \cdot g=(Z',\cA')$ is algebraic then $\cA'$ also satisfies the generic flatness property.
As a consequence, we can prove the analogous to Lemma \ref{lem_3.13T} for this case. 
Therefore, the statement follows from \cite[Lemma 9.7]{BLMNPS20}. 
\edem

\brem 
Let $X$ be a smooth projective variety of $\dim(X)>1$ and assume there is an algebraic stability condition $\sigma=(Z,\cA)\in \Stab(X)$.
We then have that there is an algebraic gluing pre-stability condition $\tau=(W,\cB)=\gl(\sigma, \dots, \sigma)$ on $\cT_{X,n}$.
Moreover, if $\cA$ satisfies generic flatness, by Corollary \ref{corollary_GFgluing} we get that $\cB$ also satisfies generic flatness. 
By Theorem \ref{thm:OHC}, the heart $\cB$ also satisfies the open heart property.
Consequently, it is enough to prove the boundedness of the moduli stack $\cM^{\beta,\phi}(\tau),$ which is expected, in order to conclude that it is algebraic of finite type over $\CC$.
See \cite{T08} for $K3$ surfaces and \cite{PT19} for 3-folds.
\erem 

\subsection{Good moduli spaces}

We will now apply the groundbreaking result \cite{AHLH19} to show the existence of the good moduli spaces as defined by Alper \cite{A13}.

The main theorem of the section is the following.
\begin{theorem}\label{thm:good-moduli}
Consider the moduli stack $\mathcal{M}^{\beta,\phi}(\sigma)$ as in Theorem \ref{thm_main1} or as in Proposition \ref{prop_algstack_chains}.
Then $\mathcal{M}^{\beta,\phi}(\sigma)$ admits a good moduli space $M^{\beta,\phi}(\sigma)$ which is an algebraic space over $\CC$.
Moreover, $M^{\beta,\phi}(\sigma)$ is proper.
\end{theorem}

\bdem 
Since $\mathcal{M}^{\beta,\phi}(\sigma)$ is an algebraic stack of finite type over $\CC$, we can follow the same steps as in \cite[Theorem 7.25]{AHLH19}. 
We obtain that $\mathcal{M}^{\beta,\phi}(\sigma)$ admits a separated good moduli space $M^{\beta,\phi}(\sigma).$ 
To prove that $M^{\beta,\phi}(\sigma)$ is proper it suffices to prove the existence part of the valuative criteria for properness \cite[Theorem A]{AHLH19}, which follows from Proposition \ref{prop:val-crit} below.
\edem

The following is an analogous result to \cite[Proposition 4.1.1]{AP06} in our setting. 
Combining the following proposition with \cite[Theorem A]{AHLH19} we deduce the properness claimed in Theorem \ref{thm:good-moduli}.

\bp\label{prop:val-crit} Let $\sigma=(Z,\cA)\in \Stab(\cT_{X,n})$ be an algebraic stability condition. 
Let $S$ be a curve and $U\coloneqq S\setminus\{p\}$ where $p\in S$ is a closed point and $j\colon U\hookrightarrow S.$
Let $E_{U}\in \cA_{U}$ such that $E_{s}\in \mathcal{P}_{\sigma}(1)$ for all $s\in U.$
Then there is an $E\in \cA_{S}$ such that $j^*(E)=E_{U}$ and $E_{s}\in \mathcal{P}_{\sigma}(1)$ for all $s\in S$.
\ep

\begin{remark}
Proposition \ref{prop:val-crit} also appears in the more general relative setting in \cite[Lemma 21.22]{BLMNPS20}.
We include a proof as a pleasant application of our methodology.
\end{remark}

\bdem 
Consider the realisation $\cT_{X,n} \subset D^b(Y_{X,n})$ and the heart $\widetilde{\cA}\subset D^b(Y_{X,n})$ as constructed in Section 4.2 such that $\cA \subset \widetilde{\cA}$.
Then $\cA_{U}\subseteq \widetilde{\cA}_{U}.$
By \cite[Lemma 3.2.1]{AP06}, there is an object $E_{0}\in \widetilde{\cA}_{S}$, such that $j^*(E_{0})=E_{U}$.
We have a triangle $E_{2}\rightarrow E_{0} \rightarrow E_{1}$ induced by the semiorthogonal decomposition $D^b(Y_{X,n}\times S)=\lin (\cT_{X,n})^{\perp}_{S},(\cT_{X,n})_{S}\rin$.

By Proposition \ref{prop:rec_base_change}, we get that $E_2\in \cA_{S}$ and moreover $j^!(E_2)=j^!(E_{0})=E_{U}$.
As $\cA$ is Noetherian there is a maximal $S$-torsion subobject $F\subseteq E_{2} $ in $\cA_{S},$ with support $\{p\}$.
See \cite[Definition 6.3]{BLMNPS20} for the definition of $S$-torsion.

We define $E\coloneqq E_{2}/F\in \cA_{S}$.
Note that  $E$ is $S$-torsion free and therefore, by \cite[Lemma 6.12]{BLMNPS20}, it is t-flat.
By the same argument of \cite[Lemma 4.1.2]{AP06}, we obtain that $E_{s}\in \mathcal{P}_{\sigma}(1)$ for all $s\in S$.
\edem 

\renewcommand{\bibname}{References}
\bibliography{bib.bib}
\bibliographystyle{acm}

\end{document}